\documentclass[10pt]{article}
   \usepackage{graphicx}
   \usepackage{epsfig}
   \usepackage{amssymb}
   \usepackage{amsmath}
   \usepackage{amsthm}
   \newtheorem{lemma}{Lemma}[section]
   \newtheorem{theorem}{Theorem}[section]

   {\end{description}}

   {\hfill $\bullet$ \\}

   \newcommand{\be}{\begin{equation}}
   \newcommand{\ee}{\end{equation}}

\begin{document}
    \title{A strong second-order two-stage explicit/implicit technique with spectral orthogonal basis Galerkin finite element method for two-dimensional Gray-Scott model}
  \author{Eric Ngondiep\thanks{\textbf{Correspondence to:} Eric Ngondiep, ericngondiep@gmail.com/engondiep@imamu.edu.sa}}
   \date{\small{Department of Mathematics and Statistics, College of Science, Imam Mohammad Ibn Saud\\ Islamic University
        (IMSIU), $90950$ Riyadh $11632,$ Saudi Arabia.}}
   \maketitle

   \textbf{Abstract.}
      This paper proposes a strong second-order two-step explicit/implicit technique with spectral orthogonal basis Galerkin finite element method for solving a two-dimensional Gray-Scott model subject to appropriate initial and boundary conditions. The constructed approach discretizes at the first stage utilizing a second-order explicit method while a second-order implicit scheme is employed at the second phase. The space derivatives are approximated with the Galerkin finite element formulation combined with a spectral orthogonal basis. With this combination, the errors increased at the first stage are balanced by the errors decreased at the second phase so that the stability is maintained. Furthermore, the use of the spectral orthogonal basis minimizes the space errors. Thus, the new computational approach calculates efficiently numerical solutions and preserves a strong stability and high-order accuracy. The theoretical studies indicate that the proposed strategy is unconditionally stable, temporal second-order accurate and spatial $qth$-order convergent using the $L^{\infty}(0,T;[L^{\infty}(\Omega)]^{2})$-norm, where $q$ is an integer greater than or equal $2$. Some numerical examples are performed to confirm the theory and to demonstrate the efficiency of the developed algorithm.\\
    \text{\,}\\

   \\
   {\bf AMS Subject Classification (MSC). 65M12, 65M15, 65M60, 65M70}.

  \section{Introduction}\label{sec1}
   A variety of physical-chemical processes are represented mathematically by the reaction-diffusion systems \cite{36tir,1en,10tir,38tir,2en,45tir}. The rate of heat and mass formation is described by the reaction term \cite{10zwz,14am} whereas the rate of heat and mass transfer is described by the diffusion term \cite{26am}. The temperature and concentration distributions are usually explained by this model with some pattern formation \cite{5tir,44tir,20zwz}. The authors of \cite{44tir,21zwz} employed numerical approaches to explain pattern generation, whereas many chemical experiments discovered the two-dimensional Turing bifurcations and Turing patterns \cite{3tir,11tir,4tir}. Consequently, the reaction-diffusion model has been used in many fields, including biology, chemistry, medicine, and more. Both the Gray-Scott (GS) model \cite{14am,7tir,10zwz} and the Gierer-Meinhardt model \cite{9zwz} are the two most popular variations which have drawn a lot of interest. However, the Gray-Scott system is well studied among chemical reaction-diffusion models. This model, developed by P. Gray and S. K. Scott in the $1980$s, is a simplified version of the Oregonator model of the Belousov-Zhabotinsky reaction \cite{10zwz}. The stability of homoclinic and Gray-Scott model heteroclinic orbit systems was studied by a few authors \cite{7tir}.  Thus, the Gray-Scott system is a prime example of emergent behavior, self-organization, and instability-driven pattern generation that goes beyond its chemical origins in physics, mathematical biology, and materials science \cite{14am,10zwz}. About mathematics' well-posedness, this model belongs to a broad class of differential equations whose mathematical properties such as: existence, uniqueness, and regularity of solutions, have been extensively studied using frameworks for weak and strong solutions \cite{42tir,3en,8tir,36tir,5en}.

   In this paper, we consider the following Gray-Scott problem defined in \cite{tir} as\\

     \begin{equation}\label{1}
     \left\{
       \begin{array}{ll}
         u_{t}-\alpha_{1}\Delta u=\beta_{0}(1-u)-uv^{2}, & \hbox{on $\Omega\times(0,\text{\,}T]$} \\
     \text{\,}\\
         v_{t}-\alpha_{2}\Delta v=-(\beta_{0}+k_{0})v+uv^{2},  & \hbox{on $\Omega\times(0,\text{\,}T]$} \\
       \end{array}
     \right.
     \end{equation}
     with initial conditions
      \begin{equation}\label{2}
      u(x,0)=u_{0}(x),\text{\,\,\,}v(x,0)=v_{0}(x),\text{\,\,\,}u_{t}(x,0)=u_{1}(x),\text{\,\,\,}v_{t}(x,0)=v_{1}(x),\text{\,\,\,\,on\,\,\,\,}\overline{\Omega}=\Omega\cup\Gamma,
     \end{equation}
     and Neumann boundary conditions
      \begin{equation}\label{3}
      \frac{\partial u}{\partial \eta}=0,\text{\,\,\,}\frac{\partial v}{\partial \eta}=0,\text{\,\,\,\,on\,\,\,\,}\Gamma\times(0,\text{\,}T],
     \end{equation}
     where $\Omega$ is a bounded domain in $\mathbb{R}^{2}$ with piecewise smooth boundary $\Gamma$, $T$ represents the final time, $\alpha_{j}>0$, $j=1,2$, are the diffusion coefficients for the chemical species $u$ and $v$, $\beta_{0}>0$, denotes the feed rate of the reactant $u$, $\beta_{0}+k_{0}>0$, represents the removal rate of the intermediate species $v$, $w_{t}$ designates $\frac{\partial{w}}{\partial{t}}$ for a function $w$, $\Delta$ means the Laplace operator  whereas: $u_{0}$, $v_{0}$, $u_{1}$, and $v_{1}$, denote the initial conditions. Furthermore, the homogeneous Neumann boundary conditions are used to simulate a closed or insulated system in which no material can enter or exit the domain. In terms of biology, this simulates situations where neither species can leave the domain, like a petri dish or a membrane-bounded cell.\\

       Setting
       \begin{equation*}
        w=(u,v)^{t}, \text{\,\,}w_{0}=(u_{0},v_{0})^{t},\text{\,\,} w_{1}=(u_{1},v_{1})^{t},\text{\,\,} \Lambda w=(\Delta u,\Delta v)^{t}, \text{\,\,} F_{1}(t,u,v)=\beta_{0}(1-u)-uv^{2},
       \end{equation*}
     \begin{equation}\label{4a}
    \text{\,\,} F_{2}(t,u,v)=-(\beta_{0}+k_{0})v+uv^{2}, \text{\,\,}F(t,w)=(F_{1}(t,w),F_{2}(t,w))^{t},
    \end{equation}
     the initial-boundary value problem $(\ref{1})$ -$(\ref{3})$ can be expressed in the vector form as
      \begin{equation}\label{4}
      w_{t}-\alpha.*\Lambda w=F(t,w), \text{\,\,\,on\,\,\,}\Omega\times(0,\text{\,}T],
     \end{equation}
     subject to initial conditions
      \begin{equation}\label{5}
      w(x,0)=w_{0}(x),\text{\,\,\,}w_{t}(x,0)=w_{1}(x),\text{\,\,\,\,on\,\,\,\,}\overline{\Omega},
     \end{equation}
     and boundary condition
      \begin{equation}\label{6}
      \frac{\partial w}{\partial \eta}=0,\text{\,\,\,\,on\,\,\,\,}\Gamma\times(0,\text{\,}T],
     \end{equation}
     where the symbol ".*" means the componentwise multiplication and $w^{t}$ denotes the transpose of the vector $w$.\\

     The construction of the new strategy in the sequel will make extensive use of the nonlinear system $(\ref{4})$ with initial-boundary conditions $(\ref{5})$ -$(\ref{6})$.\\

       The development of speckle patterns has been the subject of numerous important and useful research findings in recent years, and several researchers have developed a broad range of numerical techniques for solving a large class of  reaction-diffusion systems including the Gray-Scott problems, each motivated by the nonlinear and stiff nature of the equations. These approaches consider: radial basis function methods \cite{10tir}, radial basis functions based differential quadrature strategy \cite{9tir,22tir,19tir}, POD reduced-order model based on spectral Galerkin method \cite{13tir}, mesh method based on moving kriging element free Galerkin \cite{11tir}, implicit-explicit schemes \cite{17am,18am,4en,zwz,16am}, modified cubic B-spline differential quadrature techniques \cite{38tir}, high-fidelity simulations \cite{42tir},  adaptive mesh computations \cite{45tir}, time-stepping finite element methods (FEM) \cite{14zwz,3en}, etc... However, The quality of mesh division has a major impact on the precision and dependability of conventional numerical techniques in simulating nonlinear physical and engineering problems. These traditional methods often run into mesh distortion when dealing with complex geometric shapes and significant deformations \cite{22tir}. The Method of Lines method is the foundation of another significant field of study. This method converts the problem into a big system of stiff ordinary differential equations (ODEs) in time by discretizing the spatial variables using an appropriate method. In this approach, the stiff linear terms are treated implicitly while the nonlinear terms are handled explicitly. Thus. it creates schemes with favorable stability qualities and high-order accuracy appropriate for multidimensional systems by combining the stability benefits of implicit ODE solvers with the effectiveness of explicit nonlinearity treatment. In this paper, we propose an efficient strong second-order two-stage explicit-implicit technique with spectral orthogonal basis Galerkin finite element method (SOBGFEM) for simulating the two-dimensional Gray-Scott system $(\ref{1})$, subject to initial conditions $(\ref{2})$ and boundary conditions $(\ref{3})$. The developed technique discretizes the space derivatives using the Galerkin finite element spaces with spectral orthogonal basis whereas the time derivative is approximated with a second-order two-step explicit-implicit strategy. Moreover, the first stage computes explicitly with a second-order scheme to provide a predicted solution which is corrected at the second-stage with a second-order implicit method yielding the finial approximate solution. The new computational approach is unconditionally stable, temporal second-order convergent, spatial $q$-order accurate (where $q$ is an integer greater or equal than $3$), and takes many advantages compared to a broad range on numerical schemes proposed in the literature for solving a general set of nonlinear reaction-diffusion systems \cite{zwz,tir,17am,11tir,18am}.  The errors increased at the first-stage are balanced by the errors decreased at the second-stage so that the stability of the constructed algorithm is preserved. Additionally, the spectral orthogonal basis Galerkin finite element formulation extracts and retains the essential characteristics of approximate solution within a finite dimensional space, reducing significantly the computational costs. The approximations obtained from the spectral orthogonal basis Galerkin FEM deals with the solution and its space derivatives throughout a computational domain and thus, achieving spatial high-order accuracy \cite{13tir,11tir,9tir}. The highlights of this paper are the following:
     \begin{description}
      \item[(a)] construction of the two-stage explicit-implicit technique with spectral orthogonal basis Galerkin FEM for solving the Gray-Scott model $(\ref{1})$ subject to initial conditions $(\ref{2})$ and boundary conditions $(\ref{3})$,
      \item[(b)] stability analysis and error estimate of the developed computational technique,
      \item[(c)] numerical examples to confirm the theoretical results and to demonstrate the utility and efficiency of the proposed strategy.
     \end{description}

      In the sequel we proceed as follows. Section $\ref{sec2}$ develops the two-stage explicit-implicit method with spectral orthogonal basis Galerkin finite element procedure for simulating the nonlinear system of PDEs  $(\ref{1})$ with initial-boundary conditions $(\ref{2})$-$(\ref{3})$. In Section $\ref{sec3}$, both stability and error estimates of the proposed numerical method are deeply analyzed whereas some numerical examples are carried out in Section $\ref{sec4}$ to confirm the theoretical studies. Section $\ref{sec5}$ draws the general conclusions and presents our future works.

    \section{Development of the two-stage explicit/implicit computational approach}\label{sec2}
    This section develops a strong second-order two-stage explicit/implicit scheme combined with a spectral orthogonal basis Galerkin finite element method (SOBGFEM) for a computed solution of the Gray-Scott system $(\ref{1})$ subject to initial-boundary conditions $(\ref{2})$-$(\ref{3})$.\\

    Let $\Pi_{h}=\{K_{i},\text{\,}i=1,2,...,M\}$, where $M$ is a positive integer, be a conforming discretization of the domain $\overline{\Omega}$ which consists of triangles $K$. Here, $h=\underset{1\leq i\leq M}{\max}|K_{i}|$, where $|K|$ means the diameter of $K$. The triangulation $\Pi_{h}$ is assumed to satisfy the following properties: (i) $int(K)\neq\emptyset$, for any $K\in\Pi_{h}$; (ii) $int(K_{i})\cap int(K_{j})=\emptyset$, if $i\neq j$, whereas $K_{i}\cap K_{j}\in\{e,\emptyset\}$, for $i,j=1,...,M$, where $e$ denotes the common face/edge; (iii) the triangulation $\Pi_{\Gamma h}$ induced on $\Gamma$ is quasi-uniform.\\

     Let $\sigma=\frac{T}{N}$ be the time step, where $N$ is a positive integer, and $\top_{\sigma}=\{t_{n}=n\sigma,\text{\,}n=0,1,...,N\}$, be a uniform partition of the time interval $[0,\text{\,}T]$. For the convenience of writing, set  $w_{h}^{n}=w_{h}(x,t_{n})$ and $w^{n}=w(x,t_{n})$, be the value of the computed solution and analytical solution, respectively, at point $(x,t_{n})$.\\

    We denote $\Gamma_{i}$ be the boundary of the triangle $K_{i}$ and $n_{\Gamma_{i}}$ be the unit outward normal vector of an element $e=\Gamma_{i}\cap\Gamma_{j}$. The jump and average of a vector-valued function $U$ are designated by $[[U]]$ and $\{U\}$, respectively. Both jump and average are defined as

    \begin{equation}\label{7}
     [[U]]=(U|_{\Gamma_{i}})n_{\Gamma_{i}}+(U|_{\Gamma_{j}})n_{\Gamma_{j}}\text{\,\,\,\,and\,\,\,\,}\{U\}=\frac{1}{2}(U|_{\Gamma_{i}}+U|_{\Gamma_{j}}).
    \end{equation}

  Furthermore, when $e=K_{i}\cap\Gamma$, it holds
   \begin{equation}\label{8}
    [[U]]=(U|_{\Gamma_{i}})n_{\Gamma_{i}}\text{\,\,\,\,and\,\,\,\,}\{U\}=U|_{\Gamma_{i}}.
   \end{equation}

  We remind that if $U$ is continuous on $\Omega$, so $[[U]]=0$ and $\{U\}=U$, on $\overline{K}_{i}=K_{i}\cup\Gamma_{i}$, for $i=1,...,M$.\\

  Let $\mathcal{H}_{h}$ and $\mathcal{W}_{h}$, be the sets of all piecewise polynomials defined on the domain $\Omega$.
   \begin{equation}\label{9}
   \mathcal{H}_{h}=\{z_{h}\in H^{1}(\Omega):\text{\,}z_{h}|_{K_{i}}\in\mathcal{P}_{q+1}(K_{i}),\text{\,}[[\nabla z_{h}]]=0,\text{\,\,for\,\,}i=1,...,M\},
    \end{equation}

   where $\mathcal{P}_{q+1}(K_{i})$ represents the set of all polynomials defined on $K_{i}\subset\overline{\Omega}$, with degree not exceeding $q+1$. Thus, $\mathcal{H}_{h}$ is a finite dimensional subspace of $H^{1}(\Omega)$ with dimension $\bar{M}_{q}$ less than $\frac{M(q+3)!}{2!(q+1)!}=\frac{M(q+3)(q+2)}{2}$. Set
   \begin{equation}\label{10}
   \mathcal{W}_{h}=\mathcal{H}_{h}\times\mathcal{H}_{h}.
   \end{equation}

   Hence $\mathcal{W}_{h}$ is a finite dimensional subspace of [$H^{1}(\Omega)]^{2}$, with dimension $M_{q}=2\bar{M}_{q}$.\\

     We consider the set $\mathcal{W}$, the linear operator $\Theta$, and the scalar products: $\left(\cdot,\cdot\right)_{K_{i}}$, $\left(\cdot,\cdot\right)$, $\left(\cdot,\cdot\right)_{\nabla}$, and $\left(\cdot,\cdot\right)_{\#}$, defined as:
     \begin{equation*}
     \mathcal{W}=\{U\in [H^{1}(\Omega)]^{2}:\text{\,}[[U]]=0,\text{\,\,on\,\,}e\in\{K_{i}\cap K_{j}\neq\emptyset,\text{\,}K_{i}\cap\Gamma\neq\emptyset\}, \text{\,\,for\,\,}i,j=1,...,M\},
     \end{equation*}
     \begin{equation*}
     \Theta U=(\nabla u_{1},\nabla u_{2}),\text{\,}\left(u_{1},v_{1}\right)_{K_{i}}=\int_{K_{i}}u_{1}v_{1}dx,\text{\,} \left(U,V\right)_{K_{i}}=\int_{K_{i}}U^{t}Vdx,\text{\,}\left(B_{1},B_{2}\right)_{K_{i}}=\underset{l=1}{\overset{2}\sum}\underset{j=1}{\overset{2}\sum}
     \int_{K_{i}}B_{1lj}B_{2lj}dx,
     \end{equation*}
     \begin{equation}\label{11}
    \left(u_{1},v_{1}\right)=\underset{i=1}{\overset{M}\sum}\left(u_{1},v_{1}\right)_{K_{i}},\text{\,}\left(U,V\right)=\underset{i=1}{\overset{M}\sum}\left(U,V\right)_{K_{i}},\text{\,}
    \left(u_{1},v_{1}\right)_{\nabla}=\underset{i=1}{\overset{M}\sum}\left(\nabla u_{1},\nabla v_{1}\right)_{K_{i}},\text{\,}
    \left(B_{1},B_{2}\right)_{\#}=\underset{i=1}{\overset{M}\sum}\left(B_{1},B_{2}\right)_{K_{i}},
     \end{equation}
    and the bilinear form $\left(\cdot,\cdot\right)_{.}$ and $\left(\cdot,\cdot\right)_{\#,.}$, defined as
    \begin{equation}\label{11a}
    \left(U,V\right)_{.}=\begin{pmatrix}
                           \left(u_{1},v_{1}\right) \\
                           \text{\,}\\
                           \left(u_{2},v_{2}\right) \\
                         \end{pmatrix}
    ,\text{\,\,\,}\left(\Theta U,\Theta V\right)_{\#,.}=\begin{pmatrix}
                           \left(u_{1},v_{1}\right)_{\nabla} \\
                           \text{\,}\\
                           \left(u_{2},v_{2}\right)_{\nabla} \\
                         \end{pmatrix},
     \end{equation}
    for every vector-valued functions $U=(u_{1},u_{2})^{t}$, $V=(v_{1},v_{2})^{t}$, and $2\times2$-matrices $B_{j}$, $j=1,2$, defined on $\Pi_{h}$. The norms: $\|\cdot\|_{K_{i}}$, $\|\cdot\|$, $\|\cdot\|_{\nabla}$, and $\|\cdot\|_{\#}$ associated with these scalar products are defined as
     \begin{equation}\label{12}
    \|U\|_{K_{i}}=\sqrt{\left(U,U\right)_{K_{i}}},\text{\,\,}\|U\|=\sqrt{\left(U,U\right)},\text{\,\,}\|U\|_{\nabla}=\sqrt{\left(U,U\right)_{\nabla}},
    \text{\,\,and\,\,}\|B_{1}\|_{\#}=\sqrt{\left(B_{1},B_{1}\right)_{\#}}.
     \end{equation}

    The following integration by parts will play a crucial role in our analysis
    \begin{equation*}
    \left(\Lambda U,V\right)_{\cdot}=\underset{i=1}{\overset{M}\sum}\int_{K_{i}}(\Lambda U).*Vdx=\underset{i=1}{\overset{M}\sum}\left[\underset{e\in\Gamma_{i}}{\sum}
    \int_{e}V.*[[\Theta U]]de-\int_{K_{i}}(\Theta U).*\Theta Vdx\right]=
    \end{equation*}
     \begin{equation}\label{13}
    \underset{i=1}{\overset{M}\sum}\underset{e\in\Gamma_{i}}{\sum}\int_{e}V.*[[\Theta U]]de -\left(\Theta U,\Theta V\right)_{\#,.},
     \end{equation}
     for every $U=(u_{1},u_{2})^{t}\in[H^{2}(\Omega)]^{2}$ and $V=(v_{1},v_{2})^{t}\in[H^{1}(\Omega)]^{2}$.\\

   Let $\mathcal{B}=\{\phi_{l},\text{\,}l=1,2,...,\bar{M}_{q}\}$, be an orthogonal basis of the finite dimensional subspace $\mathcal{H}_{h}$ with respect to both scalar products $\left(\cdot,\cdot\right)$ and $\left(\cdot,\cdot\right)_{\nabla}$. Additionally, $\mathcal{B}$ is an orthonormal basis with the $L^{2}$-scalar product $\left(\cdot,\cdot\right)$. Indeed, there is a unique self-adjoint and positive definite operator $A$ from $\mathcal{H}_{h}$ to $\mathcal{H}_{h}$, satisfying for every $u_{1},v_{1}\in\mathcal{H}_{h}$, $\left(u_{1},v_{1}\right)_{\nabla}=\left(Au_{1},v_{1}\right)$. Further, the elements of $A$ are given by
    \begin{equation}\label{lp}
   (A|_{K_{i}})_{lj}=\left(\nabla \rho_{l},\nabla\phi_{j}\right)_{K_{i}},
    \end{equation}
    for $i=1,2,...,M$, and $l,j=1,2,...,\bar{M}_{q}^{[1]}$, where $\bar{M}_{q}^{[1]}$ is less than or equal $\frac{1}{2}(q+3)(q+2)$, and $\{\rho_{l},\text{\,}l=1,2,...,\bar{M}_{q}\}$ is an arbitrary basis of $\mathcal{H}_{h}$. Thus, $\mathcal{B}=\{\phi_{l},\text{\,}l=1,2,...,\bar{M}_{q}\}$ is the basis formed with the "$L^{2}$-orthonormal basis" of eigenvectors of $A$. Let $\lambda_{j}>0$, be the eigenvalue associated with the eigenvectors $\phi_{j}$. As a result, for every $v_{h}\in\mathcal{H}_{h}$, there is a unique $\bar{M}_{q}$-uplet  $(v_{h,1},v_{h,2},...,v_{h,\bar{M}_{q}})$, such that
     \begin{equation}\label{14}
      v_{h}|_{K_{i}}=\underset{l=1}{\overset{\bar{M}_{q}^{[1]}}\sum}v_{h,l}\phi_{l}|_{K_{i}},
     \end{equation}
     for $i=1,2,...,M$, where $z|_{K_{i}}$ means the restriction of the function $z$ on $K_{i}$.\\

      Suppose that $p_{h}$ is the $L^{2}$-projection from $L^{2}(\Omega)$ onto $\mathcal{H}_{h}$. Let $P_{h}=(p_{h},p_{h})$, be the orthogonal projection from $[L^{2}(\Omega)]^{2}$ onto $\mathcal{W}_{h}$. The projection $P_{h}$ is defined as: for every $U=(u_{1},u_{2})^{t}\in[L^{2}(\Omega)]^{2}$,
      \begin{equation}\label{15}
       P_{h}U=(p_{h}u_{1},p_{h}u_{2})^{t}.
     \end{equation}

     In addition, both projections satisfy
    \begin{equation}\label{16}
    \left(p_{h}u_{1},u_{1h}\right)=\left(u_{1},u_{1h}\right),\text{\,} \left(P_{h}U,U_{h}\right)=\left(U,U_{h}\right),
     \end{equation}
    for every $u_{1h}\in \mathcal{H}_{h}$, and $U_{h}\in\mathcal{W}_{h}$.\\

    To construct the first stage of the second-order explicit/implicit computational approach, we should integrate equation $(\ref{4})$ over the interval $[t_{n},t_{n+\frac{1}{2}}]$, while the integrant must be approximated at $t_{n-\frac{1}{2}}$ and $t_{n}$. The integration of equation $(\ref{4})$ from $t_{n}$ to $t_{n+\frac{1}{2}}$ gives
    \begin{equation*}
     w^{n+\frac{1}{2}}-w^{n}-\int_{t_{n}}^{t_{n+\frac{1}{2}}}\alpha.*\Lambda wdt=\int_{t_{n}}^{t_{n+\frac{1}{2}}}F(t,w)dt.
     \end{equation*}

     But, direct calculations provide
     \begin{equation}\label{17}
     w^{n+\frac{1}{2}}-w^{n}=\alpha.*\int_{t_{n}}^{t_{n+\frac{1}{2}}}\Lambda wdt+\int_{t_{n}}^{t_{n+\frac{1}{2}}}F(t,w)dt.
     \end{equation}

     The approximation of the functions $\Lambda w(t)$ and $F(t,w(t))$ at the points $t_{n-\frac{1}{2}}$ and $t_{n}$ yields
     \begin{equation*}
     \Lambda w(t)=\frac{2}{\sigma}[(t-t_{n-\frac{1}{2}})\Lambda w^{n}-(t-t_{n})\Lambda w^{n-\frac{1}{2}}]+\frac{1}{2}(t-t_{n-\frac{1}{2}})(t-t_{n})\Lambda w_{2t}(\theta_{1}(t)),
       \end{equation*}
     \begin{equation*}
     F(t,w(t))=\frac{2}{\sigma}[(t-t_{n-\frac{1}{2}})F(t_{n},w^{n})-(t-t_{n})F(t_{n-\frac{1}{2}},w^{n-\frac{1}{2}})]+
     \frac{1}{2}(t-t_{n-\frac{1}{2}})(t-t_{n})\frac{d^{2}}{dt^{2}}[F(t,w)](\theta_{2}(t)),
       \end{equation*}
     where $\theta_{j}(t)$, $j=1,2$, are between the minimum and maximum of $t_{n-\frac{1}{2}}$, $t_{n}$, and $t$. Utilizing the weighted mean value theorem for integrals, it holds
    \begin{equation}\label{18}
     \int_{t_{n}}^{t_{n+\frac{1}{2}}}\Lambda w(t)dt=\frac{\sigma}{4}(3\Lambda w^{n}-\Lambda w^{n-\frac{1}{2}})+\frac{5\sigma^{3}}{96}\Lambda w_{2t}(\theta_{1}^{n}),
     \end{equation}
     \begin{equation}\label{19}
     \int_{t_{n}}^{t_{n+\frac{1}{2}}}F(t,w(t))dt=\frac{\sigma}{4}(3F(t_{n},w^{n})-F(t_{n-\frac{1}{2}},w^{n-\frac{1}{2}}))+\frac{5\sigma^{3}}{96}\frac{d^{2}}{dt^{2}}[F(t,w)]
     (\theta_{2}^{n}),
     \end{equation}
     where $\theta_{l}^{n}\in(t_{n-\frac{1}{2}},t_{n+\frac{1}{2}})$, for $l=1,2$. Substituting equations $(\ref{18})$ and $(\ref{19})$ into $(\ref{17})$, and rearranging terms result in
     \begin{equation*}
     w^{n+\frac{1}{2}}-\frac{\sigma}{4}\alpha.*\Lambda(3w^{n}-w^{n-\frac{1}{2}})=w^{n}+\frac{\sigma}{4}(3F_{n}- F_{n-\frac{1}{2}})+
     \frac{5\sigma^{3}}{96}\left(\alpha.*\Lambda w_{2t}(\theta_{1}^{n})+\frac{d^{2}}{dt^{2}}[F(t,w)](\theta_{2}^{n})\right),
     \end{equation*}
     where we set $F_{s}=F(t_{s},w^{s})$, for $s\in\{n,n-\frac{1}{2}\}$.\\

    Left multiplication of this equation by $V\in[H^{1}(\Omega)]^{2}$, integrating over $K_{i}$, summing the obtained equation, for $i=1,2,...,M$, and using the bilinear operators $\left(\cdot,\cdot\right)_{.}$ and $\left(\cdot,\cdot\right)_{\#,.}$, defined by equation $(\ref{11a})$, together with the integration by parts $(\ref{13})$, and performing simple calculations to get
    \begin{equation*}
     \left(w^{n+\frac{1}{2}},V\right)_{.}+\frac{\sigma}{4}\left(\alpha.*\Theta(3w^{n}-w^{n-\frac{1}{2}}),\Theta V\right)_{\#,.}=\left(w^{n},V\right)_{.}+\frac{\sigma}{4}\left(
     \left(3F_{n}-F_{n-\frac{1}{2}},V\right)_{.}+\right.
     \end{equation*}
   \begin{equation}\label{20}
   \left.\underset{i=1}{\overset{M}\sum}\underset{e\in\Gamma_{i}}{\overset{M}\sum}\int_{e}\alpha.*[[\Theta(3w^{n}-w^{n-\frac{1}{2}})]]Vde\right)+
   \frac{5\sigma^{3}}{96}\left(\alpha.*\Lambda w_{2t}(\theta_{1}^{n})+\frac{d^{2}}{dt^{2}}[F(t,w)](\theta_{2}^{n}),V\right)_{.}.
     \end{equation}

    Since $w\in H^{3}(0,T;\text{\,}[H^{q+1}(\Omega)]^{2})$, where $q\geq2$, so $\Theta w$ is continuous on $\Omega$. Then $[[\Theta(3w^{n}-w^{n-\frac{1}{2}})]]=3[[\Theta w^{n}]]+[[\Theta w^{n-\frac{1}{2}}]]=0$. Thus equation $(\ref{20})$ is equivalent to
    \begin{equation*}
     \left(w^{n+\frac{1}{2}},V\right)_{.}+\frac{\sigma}{4}\left(\alpha.*\Theta(3w^{n}-w^{n-\frac{1}{2}}),\Theta V\right)_{\#,.}=\left(w^{n},V\right)_{.}+\frac{\sigma}{4}
     \left(3F_{n}-F_{n-\frac{1}{2}},V\right)_{.}+
     \end{equation*}
   \begin{equation}\label{20a}
    \frac{5\sigma^{3}}{96}\left(\alpha.*\Lambda w_{2t}(\theta_{1}^{n})+\frac{d^{2}}{dt^{2}}[F(t,w)](\theta_{2}^{n}),V\right)_{.},\text{\,\,\,\,\,}V\in[H^{1}(\Omega)]^{2}.
     \end{equation}

    Omitting the error term: $\frac{5\sigma^{3}}{96}\left(\alpha.*\Lambda w_{2t}(\theta_{1}^{n})+\frac{d^{2}}{dt^{2}}[F(t,w)](\theta_{2}^{n}),V\right)_{.}$, replacing the exact solution $w\in H^{3}(0,T;\text{\,}[H^{q+1}(\Omega)]^{2})$, with the computed solution $w_{h}(t)=(u_{h}(t),v_{h}(t))^{t}=\left(\underset{l=1}{\overset{\bar{M}_{q}}\sum}u_{h,l}(t)\phi_{l}, \underset{l=1}{\overset{\bar{M}_{q}}\sum}v_{h,l}(t)\phi_{l}\right)^{t} \in \mathcal{W}_{h}$, for $t\in[0,\text{\,}T]$, where $u_{h,l}(t)$ (resp., $v_{h,l}(t)$), for $l=1,2,...,\bar{M}_{q}$, are the unique components of $u_{h}(t)$ (resp., $v_{h}(t)$), in the orthogonal basis $\mathcal{B}=\{\phi_{l}:\text{\,}l=1,...,\bar{M}_{q}\}$ (according to equation $(\ref{14})$), approximation $(\ref{20a})$ becomes
    \begin{equation*}
     \left(w_{h}^{n+\frac{1}{2}},V\right)_{.}+\frac{\sigma}{4}\left(\alpha.*\Theta(3w_{h}^{n}-w_{h}^{n-\frac{1}{2}}),\Theta V\right)_{\#,.}=\left(w_{h}^{n},V\right)_{.}+\frac{\sigma}{4}\left(3F(t_{n},w_{h}^{n})-F(t_{n-\frac{1}{2}},w_{h}^{n-\frac{1}{2}}),V\right)_{.},
     \end{equation*}
    which is equivalent to
    \begin{equation*}
     \underset{l=1}{\overset{\bar{M}_{q}}\sum}u_{h,l}^{n+\frac{1}{2}}\left(\phi_{l},z\right)+\frac{\alpha_{1}\sigma}{4}\underset{l=1}{\overset{\bar{M}_{q}}\sum}
     (3u_{h,l}^{n}-u_{h,l}^{n-\frac{1}{2}})\left(\phi_{l},z\right)_{\nabla}=\underset{l=1}{\overset{\bar{M}_{q}}\sum}u_{h,l}^{n}\left(\phi_{l},z\right)
     +\frac{\sigma}{4}\left(3F_{1}\left(t_{n},\underset{l=1}{\overset{\bar{M}_{q}}\sum}u_{h,l}^{n}\phi_{l},\underset{l=1}{\overset{\bar{M}_{q}}\sum}v_{h,l}^{n}\phi_{l}\right)-
     \right.
     \end{equation*}
     \begin{equation}\label{21}
     \left.F_{1}\left(t_{n-\frac{1}{2}},\underset{l=1}{\overset{\bar{M}_{q}}\sum}u_{h,l}^{n-\frac{1}{2}}\phi_{l},\underset{l=1}{\overset{\bar{M}_{q}}\sum}
     v_{h,l}^{n-\frac{1}{2}}\phi_{l}\right),z\right),\text{\,\,\,\,\,\,\,\,\,\,}z\in H^{1}(\Omega),
     \end{equation}
     \begin{equation*}
     \underset{l=1}{\overset{\bar{M}_{q}}\sum}v_{h,l}^{n+\frac{1}{2}}\left(\phi_{l},z\right)+\frac{\alpha_{2}\sigma}{4}\underset{l=1}{\overset{\bar{M}_{q}}\sum}
     (3v_{h,l}^{n}-v_{h,l}^{n-\frac{1}{2}})\left(\phi_{l},z\right)_{\nabla}=\underset{l=1}{\overset{\bar{M}_{q}}\sum}v_{h,l}^{n}\left(\phi_{l},z\right)
     +\frac{\sigma}{4}\left(3F_{2}\left(t_{n},\underset{l=1}{\overset{\bar{M}_{q}}\sum}u_{h,l}^{n}\phi_{l},\underset{l=1}{\overset{\bar{M}_{q}}\sum}v_{h,l}^{n}\phi_{l}\right)-
     \right.
     \end{equation*}
     \begin{equation}\label{22}
     \left.F_{2}\left(t_{n-\frac{1}{2}},\underset{l=1}{\overset{\bar{M}_{q}}\sum}u_{h,l}^{n-\frac{1}{2}}\phi_{l},\underset{l=1}{\overset{\bar{M}_{q}}\sum}
     v_{h,l}^{n-\frac{1}{2}}\phi_{l}\right),z\right),\text{\,\,\,\,\,\,\,\,\,\,}z\in H^{1}(\Omega),
     \end{equation}

     Because $\mathcal{B}=\{\phi_{l}:\text{\,}l=1,...,\bar{M}_{q}\}$ is an orthogonal basis of $\mathcal{H}_{h}$, with respect to both scalar products $\left(\cdot,\cdot\right)$ and $\left(\cdot,\cdot\right)_{\nabla}$, that satisfies $\left(\phi_{l},\phi_{j}\right)=\delta_{lj}$ and $\left(\phi_{l},\phi_{j}\right)_{\nabla}=\lambda_{j}\delta_{lj}$, for $l,j=1,...,\bar{M}_{q}$, where $\delta_{lj}=1$, if $l=j$, and $\delta_{lj}=0$, whenever $l\neq j$, and $\lambda_{j}$ are the eigenvalues of the linear operator $A$ defined in equation $(\ref{lp})$. Taking $z=\phi_{j}$, approximations $(\ref{21})$-$(\ref{22})$ imply
     \begin{equation*}
     u_{h,j}^{n+\frac{1}{2}}+\frac{\alpha_{1}\lambda_{j}\sigma}{4}(3u_{h,j}^{n}-u_{h,j}^{n-\frac{1}{2}})=u_{h,j}^{n}
     +\frac{\sigma}{4}\left(3F_{1}\left(t_{n},\underset{l=1}{\overset{\bar{M}_{q}}\sum}u_{h,l}^{n}\phi_{l},\underset{l=1}{\overset{\bar{M}_{q}}\sum}v_{h,l}^{n}\phi_{l}\right)-
     \right.
     \end{equation*}
     \begin{equation}\label{23}
     \left.F_{1}\left(t_{n-\frac{1}{2}},\underset{l=1}{\overset{\bar{M}_{q}}\sum}u_{h,l}^{n-\frac{1}{2}}\phi_{l},\underset{l=1}{\overset{\bar{M}_{q}}\sum}
     v_{h,l}^{n-\frac{1}{2}}\phi_{l}\right),\phi_{j}\right),
     \end{equation}
     \begin{equation*}
     v_{h,j}^{n+\frac{1}{2}}+\frac{\alpha_{2}\lambda_{j}\sigma}{4}(3v_{h,j}^{n}-v_{h,j}^{n-\frac{1}{2}})=v_{h,j}^{n}
     +\frac{\sigma}{4}\left(3F_{2}\left(t_{n},\underset{l=1}{\overset{\bar{M}_{q}}\sum}u_{h,l}^{n}\phi_{l},\underset{l=1}{\overset{\bar{M}_{q}}\sum}v_{h,l}^{n}\phi_{l}\right)-
     \right.
     \end{equation*}
     \begin{equation}\label{24}
     \left.F_{2}\left(t_{n-\frac{1}{2}},\underset{l=1}{\overset{\bar{M}_{q}}\sum}u_{h,l}^{n-\frac{1}{2}}\phi_{l},\underset{l=1}{\overset{\bar{M}_{q}}\sum}
     v_{h,l}^{n-\frac{1}{2}}\phi_{l}\right),\phi_{j}\right),
     \end{equation}
     for $j=1,2,...,\bar{M}_{q}$. Equations $(\ref{21})$-$(\ref{22})$ denote the first stage of the desired algorithm. It is not hard to observe that these formulations works with an explicit numerical method.\\

     The description of the second step of the explicit-implicit numerical scheme requires the integration of equation $(\ref{4})$ from $t_{n+\frac{1}{2}}$ to $t_{n+1}$, and the interpolation of the integrants at points $t_{n+\frac{1}{2}}$ and $t_{n+1}$. Thus, it holds
     \begin{equation*}
     w^{n+1}-w^{n+\frac{1}{2}}-\int_{t_{n+\frac{1}{2}}}^{t_{n+1}}\alpha.*\Lambda wdt=\int_{t_{n+\frac{1}{2}}}^{t_{n+1}}F(t,w)dt.
     \end{equation*}
     This is equivalent to
     \begin{equation}\label{25}
     w^{n+1}-\alpha.*\int_{t_{n+\frac{1}{2}}}^{t_{n+1}}\Lambda wdt=w^{n+\frac{1}{2}}+\int_{t_{n+\frac{1}{2}}}^{t_{n+1}}F(t,w)dt.
     \end{equation}

     But, straightforward computations give
    \begin{equation}\label{26}
    \int_{t_{n+\frac{1}{2}}}^{t_{n+1}}\Lambda w(t)dt=\frac{\sigma}{4}(\Lambda w^{n+1}+\Lambda w^{n+\frac{1}{2}})-\frac{\sigma^{3}}{96}\Lambda w_{2t}(\theta_{3}^{n+1}),
     \end{equation}
     \begin{equation}\label{27}
     \int_{t_{n+\frac{1}{2}}}^{t_{n+1}}F(t,w(t))dt=\frac{\sigma}{4}(F_{n+1}+F_{n+\frac{1}{2}})-\frac{\sigma^{3}}{96}\frac{d^{2}}{dt^{2}}[F(t,w)]
     (\theta_{4}^{n+1}),
     \end{equation}
     where $\theta_{l}^{n+1}\in(t_{n+\frac{1}{2}},t_{n+1})$, for $l=3,4$. Plugging approximations $(\ref{25})$-$(\ref{27})$ to get
     \begin{equation*}
     w^{n+1}-\frac{\sigma}{4}\alpha.*\Lambda(w^{n+1}+w^{n+\frac{1}{2}})=w^{n+\frac{1}{2}}+\frac{\sigma}{4}(F_{n+1}+F_{n+\frac{1}{2}})-
     \frac{\sigma^{3}}{96}\left(\alpha.*\Lambda w_{2t}(\theta_{3}^{n+1})+\frac{d^{2}}{dt^{2}}[F(t,w)](\theta_{4}^{n+1})\right).
     \end{equation*}

    Multiplying both sides of this equation by $V\in[H^{1}(\Omega)]^{2}$, utilizing the bilinear operators $\left(\cdot,\cdot\right)_{.}$ and $\left(\cdot,\cdot\right)_{\#,.}$, and performing direct computations, we obtain
    \begin{equation*}
     \left(w^{n+1},V\right)_{.}+\frac{\sigma}{4}\left(\alpha.*\Theta(w^{n+1}+w^{n+\frac{1}{2}}),\Theta V\right)_{\#,.}=\left(w^{n+\frac{1}{2}},V\right)_{.}+ \frac{\sigma}{4}\left(F_{n+1}+F_{n+\frac{1}{2}},V\right)_{.}-
     \end{equation*}
   \begin{equation}\label{27a}
   \frac{\sigma^{3}}{96}\left(\alpha.*\Lambda w_{2t}(\theta_{3}^{n+1})+\frac{d^{2}}{dt^{2}}[F(t,w)](\theta_{4}^{n+1}),V\right)_{.}.
     \end{equation}

     Truncating the error term: $-\frac{\sigma^{3}}{96}\left(\alpha.*\Lambda w_{2t}(\theta_{3}^{n+1})+\frac{d^{2}}{dt^{2}}[F(t,w)](\theta_{4}^{n+1}),V\right)_{.}$, replacing the analytical solution $w\in H^{3}(0,T;\text{\,}[H^{q+1}(\Omega)]^{2})$, with the approximate one $w_{h}(t)=(u_{h}(t),v_{h}(t))^{t}= \left(\underset{l=1}{\overset{\bar{M}_{q}}\sum}u_{h,l}(t)\phi_{l},\underset{l=1}{\overset{\bar{M}_{q}}\sum}v_{h,l}(t)\phi_{l}\right)^{t} \in \mathcal{W}_{h}$, to obtain
    \begin{equation*}
     \underset{l=1}{\overset{\bar{M}_{q}}\sum}u_{h,l}^{n+1}\left(\phi_{l},z\right)+\frac{\alpha_{1}\sigma}{4}\underset{l=1}{\overset{\bar{M}_{q}}\sum}
     (u_{h,l}^{n+1}+u_{h,l}^{n+\frac{1}{2}})\left(\phi_{l},z\right)_{\nabla}=\underset{l=1}{\overset{\bar{M}_{q}}\sum}u_{h,l}^{n+\frac{1}{2}}\left(\phi_{l},z\right)
     +\frac{\sigma}{4}\left(F_{1}\left(t_{n+1},\underset{l=1}{\overset{\bar{M}_{q}}\sum}u_{h,l}^{n+1}\phi_{l},\underset{l=1}{\overset{\bar{M}_{q}}\sum}v_{h,l}^{n+1}\phi_{l}\right)+
     \right.
     \end{equation*}
     \begin{equation}\label{28}
     \left.F_{1}\left(t_{n+\frac{1}{2}},\underset{l=1}{\overset{\bar{M}_{q}}\sum}u_{h,l}^{n+\frac{1}{2}}\phi_{l},\underset{l=1}{\overset{\bar{M}_{q}}\sum}
     v_{h,l}^{n+\frac{1}{2}}\phi_{l}\right),z\right),
     \end{equation}
     \begin{equation*}
     \underset{l=1}{\overset{\bar{M}_{q}}\sum}v_{h,l}^{n+1}\left(\phi_{l},z\right)+\frac{\alpha_{2}\sigma}{4}\underset{l=1}{\overset{\bar{M}_{q}}\sum}
     (v_{h,l}^{n+1}+v_{h,l}^{n+\frac{1}{2}})\left(\phi_{l},z\right)_{\nabla}=\underset{l=1}{\overset{\bar{M}_{q}}\sum}v_{h,l}^{n+\frac{1}{2}}\left(\phi_{l},z\right)
     +\frac{\sigma}{4}\left(F_{2}\left(t_{n+1},\underset{l=1}{\overset{\bar{M}_{q}}\sum}u_{h,l}^{n+1}\phi_{l},\underset{l=1}{\overset{\bar{M}_{q}}\sum}v_{h,l}^{n+1}\phi_{l}\right)+
     \right.
     \end{equation*}
     \begin{equation}\label{29}
     \left.F_{2}\left(t_{n+\frac{1}{2}},\underset{l=1}{\overset{\bar{M}_{q}}\sum}u_{h,l}^{n+\frac{1}{2}}\phi_{l},\underset{l=1}{\overset{\bar{M}_{q}}\sum}
     v_{h,l}^{n+\frac{1}{2}}\phi_{l}\right),z\right).
     \end{equation}

      Taking $z=\phi_{j}$, and utilizing the fact that $\mathcal{B}=\{\phi_{l}:\text{\,}l=1,...,\bar{M}_{q}\}$ is an orthonormal basis, with respect to the scalar product $\left(\cdot,\cdot\right)$ and an orthogonal basis for the inner product $\left(\cdot,\cdot\right)_{\nabla}$, both equations $(\ref{28})$-$(\ref{29})$ become
     \begin{equation*}
     u_{h,j}^{n+1}+\frac{\alpha_{1}\lambda_{j}\sigma}{4}(u_{h,j}^{n+1}+u_{h,j}^{n+\frac{1}{2}})=u_{h,j}^{n+\frac{1}{2}}
     +\frac{\sigma}{4}\left(F_{1}\left(t_{n+1},\underset{l=1}{\overset{\bar{M}_{q}}\sum}u_{h,l}^{n+1}\phi_{l},\underset{l=1}{\overset{\bar{M}_{q}}\sum}v_{h,l}^{n+1}\phi_{l}\right)+
     \right.
     \end{equation*}
     \begin{equation}\label{30}
     \left.F_{1}\left(t_{n+\frac{1}{2}},\underset{l=1}{\overset{\bar{M}_{q}}\sum}u_{h,l}^{n+\frac{1}{2}}\phi_{l},\underset{l=1}{\overset{\bar{M}_{q}}\sum}
     v_{h,l}^{n+\frac{1}{2}}\phi_{l}\right),\phi_{j}\right),
     \end{equation}
     \begin{equation*}
     v_{h,j}^{n+1}+\frac{\alpha_{2}\lambda_{j}\sigma}{4}(v_{h,j}^{n+1}+v_{h,j}^{n+\frac{1}{2}})=v_{h,j}^{n+\frac{1}{2}}
     +\frac{\sigma}{4}\left(F_{2}\left(t_{n+1},\underset{l=1}{\overset{\bar{M}_{q}}\sum}u_{h,l}^{n+1}\phi_{l},\underset{l=1}{\overset{\bar{M}_{q}}\sum}v_{h,l}^{n+1}\phi_{l}\right)+
     \right.
     \end{equation*}
     \begin{equation}\label{31}
     \left.F_{2}\left(t_{n+\frac{1}{2}},\underset{l=1}{\overset{\bar{M}_{q}}\sum}u_{h,l}^{n+\frac{1}{2}}\phi_{l},\underset{l=1}{\overset{\bar{M}_{q}}\sum}
     v_{h,l}^{n+\frac{1}{2}}\phi_{l}\right),\phi_{j}\right),
     \end{equation}
     for $j=1,2,...,\bar{M}_{q}$. Equations $(\ref{30})$-$(\ref{31})$ represents the second stage of the constructed computational approach. It is easy to see that these formulations works with an implicit numerical scheme.\\

   To start the proposed approach given by equations $(\ref{23})$-$(\ref{24})$ and $(\ref{30})$-$(\ref{31})$, the terms $u_{h,l}^{0}$, $v_{h,l}^{0}$, $u_{h,l}^{\frac{1}{2}}$ and $v_{h,l}^{\frac{1}{2}}$, for $l=1,2,...,\bar{M}_{q}$, are needed. However, the terms $u_{h,l}^{0}$ and $v_{h,l}^{0}$, can be directly obtained from the initial condition $(\ref{2})$ and using the $L^{2}$-projection defined by equation $(\ref{16})$. That is,
   \begin{equation}\label{32}
   u_{h,l}^{0}=(p_{h}u_{0})_{l},\text{\,\,\,\,}v_{h,l}^{0}=(p_{h}v_{0})_{l}, \text{\,\,\,for\,\,\,\,} l=1,2,...,\bar{M}_{q}.
     \end{equation}

   Furthermore, the application of the Taylor series gives
     \begin{equation*}
   w^{\frac{1}{2}}=w_{0}+\frac{\sigma}{2}w_{t}^{0}+\frac{\sigma^{2}}{8}w_{2t}(\theta_{0}),
     \end{equation*}
   where $\theta_{0}\in(0,\frac{\sigma}{2})$. Using the initial condition $(\ref{5})$, this is equivalent to
   \begin{equation}\label{33a}
   w^{\frac{1}{2}}=w_{0}+\frac{\sigma}{2}w_{1}+\frac{\sigma^{2}}{8}w_{2t}(\theta_{0}).
     \end{equation}

     Tracking the error term $\frac{\sigma^{2}}{8}w_{2t}(\theta_{0})$, this provides
     \begin{equation}\label{33b}
   \tilde{w}^{\frac{1}{2}}=w_{0}+\frac{\sigma}{2}w_{1}=(u_{0}+\frac{\sigma}{2}u_{1},v_{0}+\frac{\sigma}{2}v_{1})^{t}.
     \end{equation}

   Takes
   \begin{equation}\label{33c}
   u_{h,j}^{\frac{1}{2}}=(p_{h}(u_{0}+\frac{\sigma}{2}u_{1}))_{j}, \text{\,\,\,}v_{h,j}^{\frac{1}{2}}=(p_{h}(v_{0}+\frac{\sigma}{2}v_{1}))_{j},\text{\,\,\,for\,\,\,\,} j=1,2,...,\bar{M}_{q}.
     \end{equation}

   A combination of equations $(\ref{23})$-$(\ref{24})$, $(\ref{30})$-$(\ref{32})$ and $(\ref{33c})$, provides the desired strong second-order two-stage explicit/implicit approach with spectral orthogonal basis Galerkin FEM for simulating a two-dimensional Gray-Scott problem $(\ref{1})$ subject to initial condition $(\ref{2})$ and boundary conditions $(\ref{3})$. Moreover, for $n=1,2,...,N-1$, and $j=1,2,...,\bar{M}_{q}$,
    \begin{equation*}
     u_{h,j}^{n+\frac{1}{2}}+\frac{\alpha_{1}\lambda_{j}\sigma}{4}(3u_{h,j}^{n}-u_{h,j}^{n-\frac{1}{2}})=u_{h,j}^{n}
     +\frac{\sigma}{4}\left(3F_{1}\left(t_{n},\underset{l=1}{\overset{\bar{M}_{q}}\sum}u_{h,l}^{n}\phi_{l},\underset{l=1}{\overset{\bar{M}_{q}}\sum}v_{h,l}^{n}\phi_{l}\right)-
     \right.
     \end{equation*}
     \begin{equation}\label{s1}
     \left.F_{1}\left(t_{n-\frac{1}{2}},\underset{l=1}{\overset{\bar{M}_{q}}\sum}u_{h,l}^{n-\frac{1}{2}}\phi_{l},\underset{l=1}{\overset{\bar{M}_{q}}\sum}
     v_{h,l}^{n-\frac{1}{2}}\phi_{l}\right),\phi_{j}\right),
     \end{equation}
     \begin{equation*}
     v_{h,j}^{n+\frac{1}{2}}+\frac{\alpha_{2}\lambda_{j}\sigma}{4}(3v_{h,j}^{n}-v_{h,j}^{n-\frac{1}{2}})=v_{h,j}^{n}
     +\frac{\sigma}{4}\left(3F_{2}\left(t_{n},\underset{l=1}{\overset{\bar{M}_{q}}\sum}u_{h,l}^{n}\phi_{l},\underset{l=1}{\overset{\bar{M}_{q}}\sum}v_{h,l}^{n}\phi_{l}\right)-
     \right.
     \end{equation*}
     \begin{equation}\label{s2}
     \left.F_{2}\left(t_{n-\frac{1}{2}},\underset{l=1}{\overset{\bar{M}_{q}}\sum}u_{h,l}^{n-\frac{1}{2}}\phi_{l},\underset{l=1}{\overset{\bar{M}_{q}}\sum}
     v_{h,l}^{n-\frac{1}{2}}\phi_{l}\right),\phi_{j}\right),
     \end{equation}
    \begin{equation*}
     u_{h,j}^{n+1}+\frac{\alpha_{1}\lambda_{j}\sigma}{4}(u_{h,j}^{n+1}+u_{h,j}^{n+\frac{1}{2}})=u_{h,j}^{n+\frac{1}{2}}
     +\frac{\sigma}{4}\left(F_{1}\left(t_{n+1},\underset{l=1}{\overset{\bar{M}_{q}}\sum}u_{h,l}^{n+1}\phi_{l},\underset{l=1}{\overset{\bar{M}_{q}}\sum}v_{h,l}^{n+1}\phi_{l}\right)+
     \right.
     \end{equation*}
     \begin{equation}\label{s3}
     \left.F_{1}\left(t_{n+\frac{1}{2}},\underset{l=1}{\overset{\bar{M}_{q}}\sum}u_{h,l}^{n+\frac{1}{2}}\phi_{l},\underset{l=1}{\overset{\bar{M}_{q}}\sum}
     v_{h,l}^{n+\frac{1}{2}}\phi_{l}\right),\phi_{j}\right),
     \end{equation}
     \begin{equation*}
     v_{h,j}^{n+1}+\frac{\alpha_{2}\lambda_{j}\sigma}{4}(v_{h,j}^{n+1}+v_{h,j}^{n+\frac{1}{2}})=v_{h,j}^{n+\frac{1}{2}}
     +\frac{\sigma}{4}\left(F_{2}\left(t_{n+1},\underset{l=1}{\overset{\bar{M}_{q}}\sum}u_{h,l}^{n+1}\phi_{l},\underset{l=1}{\overset{\bar{M}_{q}}\sum}v_{h,l}^{n+1}\phi_{l}\right)+
     \right.
     \end{equation*}
     \begin{equation}\label{s4}
     \left.F_{2}\left(t_{n+\frac{1}{2}},\underset{l=1}{\overset{\bar{M}_{q}}\sum}u_{h,l}^{n+\frac{1}{2}}\phi_{l},\underset{l=1}{\overset{\bar{M}_{q}}\sum}
     v_{h,l}^{n+\frac{1}{2}}\phi_{l}\right),\phi_{j}\right),
     \end{equation}
   subject to initial conditions
    \begin{equation}\label{s5}
    u_{h,j}^{0}=(p_{h}u_{0})_{j},\text{\,\,}v_{h,j}^{0}=(p_{h}v_{0})_{j},\text{\,\,}
    u_{h,j}^{\frac{1}{2}}=(p_{h}(u_{0}+\frac{\sigma}{2}u_{1}))_{j}, \text{\,\,}v_{h,j}^{\frac{1}{2}}=(p_{h}(v_{0}+\frac{\sigma}{2}v_{1}))_{j},
     \end{equation}
   and Neumann boundary conditions
   \begin{equation}\label{s6}
    [[u_{h,j}^{n}]]=0,\text{\,\,\,\,}[[v_{h,j}^{n}]]=0,\text{\,\,\,\,\,\,\,on\,\,\,\,}\Gamma,
     \end{equation}
   for $n=1,...,N,$ and $j=1,2,...,\bar{M}_{q}$.

     \section{Stability analysis and error estimates}\label{sec3}
     This section analyzes the stability and the error estimates of the proposed computational technique $(\ref{s1})$-$(\ref{s6})$ for simulating the two-dimensional Gray-Scott equations $(\ref{1})$ with initial-boundary conditions $(\ref{2})$-$(\ref{3})$.

     \begin{theorem}\label{t1} (Strong stability analysis).
     Let $w_{h}(t)=(u_{h}(t),v_{h}(t))^{t}\in\mathcal{W}_{h}$, for $0\leq t\leq T$, be the numerical solution obtained from the two-step explicit/implicit approach with spectral orthogonal basis GFEM $(\ref{s1})$-$(\ref{s6})$. For small values of the time step $\sigma$, the developed strategy is strongly stable. Furthermore, it holds
     \begin{equation}\label{37}
     \underset{1\leq s\leq N}{\max}\|u_{h}^{s}\|\leq C_{0}\sqrt{\bar{M}_{q}},\text{\,\,\,\,}\underset{1\leq s\leq N}{\max}\|v_{h}^{s}\|\leq C_{0}\sqrt{\bar{M}_{q}},
     \end{equation}
     where $s$ varies in the range: $0,\frac{1}{2},1,...,N$, and $C_{0}$ is a positive constant independent of the time step $\sigma$.
     \end{theorem}

     \begin{proof}
     Set $\psi_{l}=(\phi_{l},\phi_{l})^{t}$, for $l=1,2,...,\bar{M}_{q}$. Since $w_{h}^{s}=(u_{h}^{s},v_{h}^{s})^{t}\in\mathcal{W}_{h}$, $\alpha=(\alpha_{1},\alpha_{2})$,
     $F\left(t_{s},\underset{l=1}{\overset{\bar{M}_{q}}\sum}w_{h,l}^{s}.*\psi_{l}\right)=\\
     F\left(t_{s},\underset{l=1}{\overset{\bar{M}_{q}}\sum}u_{h,l}^{s}\phi_{l},\underset{l=1}{\overset{\bar{M}_{q}}\sum}v_{h,l}^{s}\phi_{l}\right)=
     \left(F_{1}\left(t_{s},\underset{l=1}{\overset{\bar{M}_{q}}\sum}u_{h,l}^{s}\phi_{l},\underset{l=1}{\overset{\bar{M}_{q}}\sum}v_{h,l}^{s}\phi_{l}\right),
     F_{2}\left(t_{s},\underset{l=1}{\overset{\bar{M}_{q}}\sum}u_{h,l}^{s}\phi_{l},\underset{l=1}{\overset{\bar{M}_{q}}\sum}v_{h,l}^{s}\phi_{l}\right)\right)$, combining approximations $(\ref{s1})$-$(\ref{s4})$, and rearranging terms, we obtain
     \begin{equation*}
     w_{h,j}^{n+\frac{1}{2}}-w_{h,j}^{n}+0 w_{h,j}^{n-\frac{1}{2}}=\frac{\sigma}{4}\left[-\lambda_{j}\alpha.*(3w_{h,j}^{n}-w_{h,j}^{n-\frac{1}{2}})+
     \left(3F\left(t_{n},\underset{l=1}{\overset{\bar{M}_{q}}\sum}w_{h,l}^{n}.*\psi_{l}\right)-
     \right.\right.
     \end{equation*}
     \begin{equation}\label{33}
     \left.\left.F\left(t_{n-\frac{1}{2}},\underset{l=1}{\overset{\bar{M}_{q}}\sum}w_{h,l}^{n-\frac{1}{2}}.*\psi_{l}\right),\psi_{j}\right)_{.}\right],
     \end{equation}
    \begin{equation*}
     w_{h,j}^{n+1}-w_{h,j}^{n+\frac{1}{2}}+0w_{h,j}^{n}+0w_{h,j}^{n-\frac{1}{2}}=\frac{\sigma}{4}\left[-\lambda_{j}\alpha.*(w_{h,j}^{n+1}+w_{h,j}^{n+\frac{1}{2}})
     +\left(F\left(t_{n+1},\underset{l=1}{\overset{\bar{M}_{q}}\sum}w_{h,l}^{n+1}.*\psi_{l}\right)+
     \right.\right.
     \end{equation*}
     \begin{equation}\label{34}
     \left.\left.F\left(t_{n+\frac{1}{2}},\underset{l=1}{\overset{\bar{M}_{q}}\sum}w_{h,l}^{n+\frac{1}{2}}.*\psi_{l}\right),\psi_{j}\right)_{.}\right],
     \end{equation}
     for $j=1,2,...,\bar{M}_{q}$. Setting $\bar{W}_{h}^{s}=[w_{h,1}^{s},w_{h,2}^{s},...,w_{h,\bar{M}_{q}}^{s}]^{t}$ and
     \begin{equation*}
      \bar{F}(t_{s},\bar{W}_{h}^{s})=[-\lambda_{1}\alpha.*w_{h,1}^{s}+\left(F\left(t_{s},\underset{l=1}{\overset{\bar{M}_{q}}\sum}w_{h,l}^{s}.*\psi_{l}\right),\psi_{1}\right)_{.},...,
      -\lambda_{\bar{M}_{q}}\alpha.*w_{h,\bar{M}_{q}}^{s}+\left(F\left(t_{s},\underset{l=1}{\overset{\bar{M}_{q}}\sum}w_{h,l}^{s}.*\psi_{l}\right),\psi_{\bar{M}_{q}}\right)_{.}]^{t},
        \end{equation*}
        for $s\in\{n-\frac{1}{2},n,n+\frac{1}{2},n+1\}$, equations $(\ref{33})$ and $(\ref{34})$ become
     \begin{equation}\label{35}
     \bar{W}_{h}^{n+\frac{1}{2}}-\bar{W}_{h}^{n}+0 \bar{W}_{h}^{n-\frac{1}{2}}=\frac{\sigma}{4}(3\bar{F}(t_{n},\bar{W}_{h}^{n})-\bar{F}(t_{n-\frac{1}{2}},\bar{W}_{h}^{n-\frac{1}{2}})),
     \end{equation}
      \begin{equation}\label{36}
     \bar{W}_{h}^{n+1}-\bar{W}_{h}^{n+\frac{1}{2}}+0 \bar{W}_{h}^{n}+0\bar{W}_{h}^{n-\frac{1}{2}}=\frac{\sigma}{4}(\bar{F}(t_{n+1},\bar{W}_{h}^{n+1})+\bar{F}(t_{n+\frac{1}{2}},\bar{W}_{h}^{n+\frac{1}{2}})+
     0\bar{F}(t_{n},\bar{W}_{h}^{n})+0\bar{F}(t_{n-\frac{1}{2}},\bar{W}_{h}^{n-\frac{1}{2}})).
     \end{equation}

     It is not difficult to observe that equation $(\ref{35})$ is an explicit two-step method whereas equation $(\ref{36})$ is an implicit three-step method. The left and right characteristic polynomials of equation $(\ref{35})$ are given by
     \begin{equation*}
      p(\lambda)=\lambda^{2}-\lambda=\lambda(\lambda-1)\text{\,\,\,and\,\,\,}q(\lambda)=3\lambda-1,
        \end{equation*}
      while the left and right characteristic polynomials associated with equation $(\ref{36})$ are given by
      \begin{equation*}
      \bar{p}(\lambda)=\lambda^{3}-\lambda^{2}=\lambda^{2}(\lambda-1)\text{\,\,\,and\,\,\,}\bar{q}(\lambda)=\lambda^{3}+\lambda^{2}=\lambda^{2}(\lambda+1),
        \end{equation*}
        where we set $\lambda=\mu^{\frac{1}{2}}$, $\mu$ is a complex number. The roots of $p(\lambda)$ are $\lambda_{1}=1$ with multiplicity equals $1$ and $\lambda_{2}=0$ with multiplicity equals $1$, whereas that of $q(\lambda)$ is $\lambda_{3}=\frac{1}{3}$ with multiplicity equals $1$. Thus, the difference scheme $(\ref{35})$ satisfies the root condition. Therefore, this scheme is strongly stable. In addition, The roots of $\bar{p}(\lambda)$ (respectively, $\bar{q}(\lambda)$) are $\lambda_{1}=1$ with multiplicity equals $1$ and $\lambda_{2}=0$ with multiplicity equals $2$ (respectively, $\lambda_{1}=-1$ with multiplicity equals $1$ and $\lambda_{2}=0$ with multiplicity equals $2$). Also in this case, the numerical method $(\ref{36})$ satisfies the root condition and it is strongly stable.\\

        Therefore, for small values of the time step $\sigma$, the two-stage explicit/implicit method with spectral orthogonal basis Galerkin finite element method $(\ref{s1})$-$(\ref{s6})$ is strongly stable. As a result, there is a positive constant $C_{0}$ independent of the time step $\sigma$ so that
        \begin{equation*}
      \underset{0\leq s\leq N}{\max}\underset{1\leq j\leq \bar{M}_{q}}{\max}(|u_{h,j}^{s}|+|v_{h,j}^{s}|)\leq C_{0}.
        \end{equation*}

        Since $u_{h}^{s}=\underset{j=1}{\overset{\bar{M}_{q}}\sum}u_{h,j}^{s}\phi_{j}$, $v_{h}^{s}=\underset{j=1}{\overset{\bar{M}_{q}}\sum}u_{h,j}^{s}\phi_{i}$, and
        $\mathcal{B}=\{\phi_{j},\text{\,}j=1,2,...,\bar{M}_{q}\}$, is an orthonormal basis of $\mathcal{H}_{h}$ with respect to the scalar product $\left(.,.\right)$, so
        \begin{equation*}
      \|u_{h}^{s}\|=\left(\underset{j=1}{\overset{\bar{M}_{q}}\sum}|u_{h,j}^{s}|^{2}\right)^{\frac{1}{2}}\leq (\bar{M}_{q}C_{0}^{2})^{\frac{1}{2}}=\sqrt{\bar{M}_{q}}C_{0},\text{\,\,\,}\|v_{h}^{s}\|=\left(\underset{j=1}{\overset{\bar{M}_{q}}\sum}|v_{h,j}^{s}|^{2}\right)^{\frac{1}{2}}\leq (\bar{M}_{q}C_{0}^{2})^{\frac{1}{2}}=\sqrt{\bar{M}_{q}}C_{0}.
        \end{equation*}

      Taking the maximum over $s$, for $s=0,\frac{1}{2},1,...,N$, this ends the proof of Theorem $\ref{t1}$.
     \end{proof}

     \begin{theorem}\label{t2} (Error estimates).
       Suppose that $w=(u,v)^{t}\in H^{3}(0,T;\text{\,}[H^{q+1}(\Omega)]^{2})\cap H^{3}(0,T;\text{\,}\mathcal{W})$ (where $q\geq2$ is an integer) is the exact solution of the initial-boundary value problem $(\ref{1})$-$(\ref{3})$, and let $w_{h}(t)=(u_{h}(t),v_{h}(t))^{t}\in\mathcal{W}_{h}$, for $0\leq t\leq T$, be the approximate solution provided by the proposed two-stage explicit/implicit computational technique $(\ref{s1})$-$(\ref{s6})$. Set $e_{h}^{s}=w^{s}-w_{h}^{s}=(u^{s}-u_{h}^{s},v^{s}-v_{h}^{s})^{t}=(e_{uh}^{s},e_{vh}^{s})^{t}$, be the error at time $t_{s}$, for $s\in\{n-\frac{1}{2},n,n+\frac{1}{2},n+1\}$, thus it holds
     \begin{equation*}
     \|e_{uh}^{n+1}\|^{2}+\|e_{uh}^{n+\frac{1}{2}}\|^{2}+\frac{\sigma}{4}\underset{\leq l\leq 2}{\min}\{\alpha_{l}\}\left[\underset{s=\frac{1}{2}}{\overset{n}\sum}
     (3\|e_{uh}^{s}\|_{\nabla}^{2}+\|e_{uh}^{s+1}+e_{uh}^{s+\frac{1}{2}}\|_{\nabla}^{2}+\|e_{uh}^{s}-e_{uh}^{s-\frac{1}{2}}\|_{\nabla}^{2})+\right.
    \end{equation*}
    \begin{equation*}
     \left.\|e_{uh}^{n+1}\|_{\nabla}^{2}+\|e_{uh}^{n}\|_{\nabla}^{2}\right]\leq\exp(3T(1+20C_{F}^{2}))\left[2\sqrt{C_{1}}(\|u_{0}+\frac{\sigma}{2}u_{1}\|+
     \|v_{0}+\frac{\sigma}{2}v_{1}\|+\frac{1}{8}\||w_{2t}|\|_{\infty})+\right.
    \end{equation*}
     \begin{equation*}
    \left.\frac{\sqrt{C_{1}\sigma}}{2}(\underset{1\leq l\leq2}{\max}\{\alpha_{l}\}+(\frac{1}{2}+18C_{F}^{2})h^{2})^{\frac{1}{2}}(\|u_{0}\|_{q+1}+
    \|v_{0}\|_{q+1})+\sqrt{39T}\||\alpha.*\Lambda w_{2t}+\frac{d^{2}}{dt^{2}}(F(t,w))|\|_{\infty}\right]^{2}(\sigma^{2}+h^{q})^{2},
     \end{equation*}
     \begin{equation*}
     \|e_{vh}^{n+1}\|^{2}+\|e_{vh}^{n+\frac{1}{2}}\|^{2}+\frac{\sigma}{4}\underset{\leq l\leq 2}{\min}\{\alpha_{l}\}\left[\underset{s=\frac{1}{2}}{\overset{n}\sum}
     (3\|e_{vh}^{s}\|_{\nabla}^{2}+\|e_{vh}^{s+1}+e_{vh}^{s+\frac{1}{2}}\|_{\nabla}^{2}+\|e_{vh}^{s}-e_{vh}^{s-\frac{1}{2}}\|_{\nabla}^{2})+\right.
    \end{equation*}
    \begin{equation*}
     \left.\|e_{vh}^{n+1}\|_{\nabla}^{2}+\|e_{vh}^{n}\|_{\nabla}^{2}\right]\leq\exp(3T(1+20C_{F}^{2}))\left[2\sqrt{C_{1}}(\|u_{0}+\frac{\sigma}{2}u_{1}\|+
     \|v_{0}+\frac{\sigma}{2}v_{1}\|+\frac{1}{8}\||w_{2t}|\|_{\infty})+\right.
    \end{equation*}
     \begin{equation*}
    \left.\frac{\sqrt{C_{1}\sigma}}{2}(\underset{1\leq l\leq2}{\max}\{\alpha_{l}\}+(\frac{1}{2}+18C_{F}^{2})h^{2})^{\frac{1}{2}}(\|u_{0}\|_{q+1}+
    \|v_{0}\|_{q+1})+\sqrt{39T}\||\alpha.*\Lambda w_{2t}+\frac{d^{2}}{dt^{2}}(F(t,w))|\|_{\infty}\right]^{2}(\sigma^{2}+h^{q})^{2},
     \end{equation*}
    for $n=0,1,...,N-1$, where $C_{F}$ and $C_{1}$ are positive constants independent of the space step $h$ and time step $\sigma$, $\|\cdot\|_{q+1}$ means the $H^{q+1}$-norm, $\||z|\|_{\infty}=\underset{0\leq t\leq T}{\sup}\|z(t)\|$, for every function $z\in L^{\infty}(0,T;\text{\,}[L^{2}(\Omega)]^{2})$.
       \end{theorem}

       The proof of Theorem $\ref{t2}$ uses the following Lemmas

     \begin{lemma}\label{l1}\cite{9FHN}
     For every $z\in H^{q+1}(\Omega)$ (where $q\geq2$ is an integer), the following inequality is satisfied
    \begin{equation}\label{38}
     \underset{i=1}{\overset{M}\sum}\left[\|z-p_{h}z\|_{K_{i}}^{2}+h^{2}_{K_{i}}\|\nabla(z-p_{h}z)\|_{K_{i}}^{2}\right]\leq C_{1}h^{2q+2}\|z\|_{q+1}^{2},\text{\,\,and\,\,\,\,}
     \underset{i=1}{\overset{M}\sum}\|\nabla(z-p_{h}z)\|_{K_{i}}^{2}\leq C_{1}h^{2q}\|z\|_{q+1}^{2},
    \end{equation}
     where $h_{K_{i}}$ denotes the diameter of $K_{i}\in\Pi_{h}$, and $C_{1}$ is a positive constant independent of the mesh grid $h$ and time step $\sigma$.
     \end{lemma}

     \begin{lemma}\label{l2}
      Consider that $w=(u,v)^{t}\in H^{3}(0,T;\text{\,}[H^{q+1}(\Omega)]^{2})\cap H^{3}(0,T;\text{\,}\mathcal{W})$ is the analytical solution of the initial-boundary value problem $(\ref{1})$-$(\ref{3})$, and suppose that $w_{h}(t)=(u_{h}(t),v_{h}(t))^{t}\in\mathcal{W}_{h}$, for $0\leq t\leq T$, is the approximate one provided by the numerical scheme $(\ref{s1})$-$(\ref{s6})$, thus it holds
      \begin{equation}\label{39}
      \|F_{l}(t_{s},w_{h}^{s})-F_{l}(t,w^{s})\|\leq C_{F}(\|u_{h}^{s}-u^{s}\|^{2}+\|v_{h}^{s}-v^{s}\|^{2})^{\frac{1}{2}},
       \end{equation}
      for $l=1,2$, $s\in\{n-\frac{1}{2},n,n+\frac{1}{2},n+1\}$, where the real-valued functions $F_{l}$ are defined by equation $(\ref{4a})$ and $C_{F}$, is a positive constant independent of the mesh grid $h$ and time step $\sigma$.
     \end{lemma}

      \begin{proof}
       Utilizing equations $(\ref{4a})$, it holds
       \begin{equation*}
      F_{1}(t_{s},w_{h}^{s})-F_{1}(t,w^{s})=\beta_{0}(1-u_{h}^{s})-u_{h}^{s}(v_{h}^{s})^{2}-[\beta_{0}(1-u^{s})-u^{s}(v^{s})^{2}]=-\beta_{0}(u_{h}^{s}-u^{s})
      -[u_{h}^{s}(v_{h}^{s})^{2}-u^{s}(v^{s})^{2}]=
       \end{equation*}
      \begin{equation}\label{40}
      -(\beta_{0}+(v_{h}^{s})^{2})(u_{h}^{s}-u^{s})-u^{s}(v_{h}^{s}+v^{s})(v_{h}^{s}-v^{s}),
       \end{equation}
      \begin{equation*}
      F_{2}(t_{s},w_{h}^{s})-F_{2}(t,w^{s})=-(\beta_{0}+k_{0})v_{h}^{s}+u_{h}^{s}(v_{h}^{s})^{2}-[-(\beta_{0}+k_{0})v^{s}+u^{s}(v^{s})^{2}]=-(\beta_{0}+k_{0})(v_{h}^{s}-v^{s})+
       \end{equation*}
      \begin{equation}\label{41}
      u_{h}^{s}(v_{h}^{s}+v^{s})(v_{h}^{s}-v^{s})+(v^{s})^{2}(u_{h}^{s}-u^{s})=(v^{s})^{2}(u_{h}^{s}-u^{s})+[-(\beta_{0}+k_{0})+u_{h}^{s}(v_{h}^{s}+v^{s})](v_{h}^{s}-v^{s}).
       \end{equation}

      Because, $w=(u,v)^{t}\in H^{3}(0,T;\text{\,}\mathcal{W})$, there is a positive constant $C_{2}$ independent of the time step $\sigma$ and grid space $h$ so that
      \begin{equation*}
      \|u^{s}\|\leq C_{2},\text{\,\,\,\,}\|v^{s}\|\leq C_{2}.
       \end{equation*}

     Utilizing this fact, along with the Cauchy-Schwarz inequality, estimate $(\ref{37})$, and equations $(\ref{40})$-$(\ref{41})$, straightforward computations result in
      \begin{equation*}
      \|F_{1}(t_{s},w_{h}^{s})-F_{1}(t,w^{s})\|^{2}\leq 4[(\beta_{0}^{2}|\Omega|^{2}+\bar{M}_{q}^{2}C_{0}^{4})\|u_{h}^{s}-u^{s}\|^{2}
      +C_{2}^{2}(C_{2}^{2}+\bar{M}_{q}C_{0}^{2})\|v_{h}^{s}-v^{s}\|^{2}],
       \end{equation*}
      \begin{equation*}
      \|F_{2}(t_{s},w_{h}^{s})-F_{2}(t,w^{s})\|^{2}\leq 4[\frac{1}{2}C_{2}^{4}\|u_{h}^{s}-u^{s}\|^{2}+((\beta_{0}+k_{0})^{2}|\Omega|^{2}+
      2\bar{M}_{q}C_{0}^{2}(C_{2}^{2}+\bar{M}_{q}C_{0}^{2}))\|v_{h}^{s}-v^{s}\|^{2}].
       \end{equation*}

       Taking the square root, it is not hard to see that these estimates imply
        \begin{equation*}
      \|F_{1}(t_{s},w_{h}^{s})-F_{1}(t,w^{s})\|\leq 2\max\{\beta_{0}|\Omega|+\bar{M}_{q}C_{0}^{2},
      C_{2}(C_{2}+\sqrt{\bar{M}_{q}}C_{0})\}(\|u_{h}^{s}-u^{s}\|^{2}+\|v_{h}^{s}-v^{s}\|^{2})^{\frac{1}{2}},
       \end{equation*}
      \begin{equation*}
      \|F_{2}(t_{s},w_{h}^{s})-F_{2}(t,w^{s})\|\leq 2\max\{\frac{C_{2}^{2}}{\sqrt{2}},(\beta_{0}+k_{0})|\Omega|+
      \sqrt{2\bar{M}_{q}}C_{0}(C_{2}+\sqrt{\bar{M}_{q}}C_{0})\}(\|u_{h}^{s}-u^{s}\|^{2}+\|v_{h}^{s}-v^{s}\|^{2})^{\frac{1}{2}},
       \end{equation*}
       where $|\Omega|$ designates the Lebesgue measure of the domain $\Omega$. Setting
       \begin{equation}\label{41b}
     C_{F}=2\max\left\{\max\{\beta_{0}|\Omega|+\bar{M}_{q}C_{0}^{2},C_{2}(C_{2}+\sqrt{\bar{M}_{q}}C_{0})\},\text{\,}\max\{\frac{C_{2}^{2}}{\sqrt{2}},(\beta_{0}+k_{0})|\Omega|+
      \sqrt{2\bar{M}_{q}}C_{0}(C_{2}+\sqrt{\bar{M}_{q}}C_{0})\}\right\},
    \end{equation}
    this completes the proof of Lemma $\ref{l2}$.
      \end{proof}

    \begin{proof} (of Theorem $\ref{t2}$).
    Setting: $w_{h}^{s}=(u_{h}^{s},v_{h}^{s})^{t}=\left(\underset{l=1}{\overset{\bar{M}_{q}}\sum}u_{h,l}^{s}\phi_{l},\underset{l=1}{\overset{\bar{M}_{q}}\sum}
    v_{h,l}^{s}\phi_{l}\right)^{t}$, $F(t_{s},w_{h}^{s})=(F_{1}(t_{s},w_{h}^{s}),F_{2}(t_{s},w_{h}^{s}))^{t}$ for $s\in\{n-\frac{1}{2},n,n+\frac{1}{2},n+1\}$, $\psi_{j}=(\phi_{j},\phi_{j})^{t}$, using equations $(\ref{s1})$-$(\ref{s4})$, and rearranging terms, we obtain
    \begin{equation}\label{42}
     w_{h,j}^{n+\frac{1}{2}}-w_{h,j}^{n}+\frac{\sigma\lambda_{j}}{4}\alpha.*(3w_{h,j}^{n}-w_{h,j}^{n-\frac{1}{2}})=\frac{\sigma}{4} \left(3F(t_{n},w_{h}^{n})-
     F(t_{n-\frac{1}{2}},w_{h}^{n-\frac{1}{2}}),\psi_{j}\right)_{.},
     \end{equation}
    \begin{equation}\label{43}
     w_{h,j}^{n+1}-w_{h,j}^{n+\frac{1}{2}}+\frac{\sigma\lambda_{j}}{4}\alpha.*(w_{h,j}^{n+1}+w_{h,j}^{n+\frac{1}{2}})=\frac{\sigma}{4}
     \left(F(t_{n+1},w_{h}^{n+1})+F(t_{n+\frac{1}{2}},w_{h}^{n+\frac{1}{2}}),\psi_{j}\right)_{.},
     \end{equation}
     for $j=1,2,...,\bar{M}_{q}$. Because $\mathcal{B}=\{\phi_{l}:\text{\,}l=1,2,...,\bar{M}_{q}\}$, is an orthonormal basis for the scalar product $\left(.,.\right)$, and it is an orthogonal basis for the inner product $\left(.,.\right)_{\nabla}$, replacing into equations $(\ref{42})$-$(\ref{43})$, $\psi_{j}$ with $V\in [H^{2}(\Omega)]$, using the inner products $\left(.,.\right)$ and $\left(.,.\right)_{\#}$ defined by equation $(\ref{11})$, a simple manipulation gives
     \begin{equation}\label{44}
     \left(w_{h}^{n+\frac{1}{2}}-w_{h}^{n},V\right)+\frac{\sigma}{4}\left(\alpha.*\Theta(3w_{h}^{n}-w_{h}^{n-\frac{1}{2}}),\Theta V\right)_{\#}=\frac{\sigma}{4}
     \left(3F(t_{n},w_{h}^{n})-F(t_{n-\frac{1}{2}},w_{h}^{n-\frac{1}{2}}),V\right),
     \end{equation}
      \begin{equation}\label{45}
     \left(w_{h}^{n+1}-w_{h}^{n+\frac{1}{2}},V\right)+\frac{\sigma}{4}\left(\alpha.*\Theta(w_{h}^{n+1}+w_{h}^{n+\frac{1}{2}}),\Theta V\right)_{\#}= \frac{\sigma}{4}\left(F(t_{n+1},w_{h}^{n+1})+F(t_{n+\frac{1}{2}},w_{h}^{n+\frac{1}{2}}),V\right).
     \end{equation}

     In a similar manner, performing direct calculations, equations $(\ref{20a})$ and $(\ref{27a})$ become
     \begin{equation*}
     \left(w^{n+\frac{1}{2}}-w^{n},V\right)+\frac{\sigma}{4}\left(\alpha.*\Theta(3w^{n}-w^{n-\frac{1}{2}}),\Theta V\right)_{\#}=\frac{\sigma}{4}
     \left(3F(t_{n},w^{n})-F(t_{n-\frac{1}{2}},w^{n-\frac{1}{2}}),V\right)+
     \end{equation*}
   \begin{equation}\label{46}
    \frac{5\sigma^{3}}{96}\left(\alpha.*\Lambda w_{2t}(\theta_{1}^{n})+\frac{d^{2}}{dt^{2}}[F(t,w)](\theta_{2}^{n}),V\right),
     \end{equation}
      \begin{equation*}
     \left(w^{n+1}-w^{n+\frac{1}{2}},V\right)+\frac{\sigma}{4}\left(\alpha.*\Theta(w^{n+1}+w^{n+\frac{1}{2}}),\Theta V\right)_{\#}= \frac{\sigma}{4}\left(F(t_{n+1},w^{n+1})+F(t_{n+\frac{1}{2}},w^{n+\frac{1}{2}}),V\right)-
     \end{equation*}
   \begin{equation}\label{47}
   \frac{\sigma^{3}}{96}\left(\alpha.*\Lambda w_{2t}(\theta_{3}^{n+1})+\frac{d^{2}}{dt^{2}}[F(t,w)](\theta_{4}^{n+1}),V\right).
     \end{equation}

     Subtracting equations $(\ref{46})$ and $(\ref{47})$ from approximations $(\ref{44})$ and $(\ref{45})$, respectively, and utilizing equality $e_{h}^{s}=w_{h}^{s}-w^{s}$, this yields
     \begin{equation*}
     \left(e_{h}^{n+\frac{1}{2}}-e_{h}^{n},V\right)+\frac{\sigma}{4}\left(\alpha.*\Theta(3e_{h}^{n}-e_{h}^{n-\frac{1}{2}}),\Theta V\right)_{\#}=\frac{\sigma}{4}
     \left(3(F(t_{n},w_{h}^{n})-F(t_{n},w^{n}))-\right.
     \end{equation*}
   \begin{equation}\label{48}
   \left.(F(t_{n-\frac{1}{2}},w_{h}^{n-\frac{1}{2}})-F(t_{n-\frac{1}{2}},w^{n-\frac{1}{2}})),V\right)-\frac{5\sigma^{3}}{96}\left(\alpha.*\Lambda w_{2t}(\theta_{1}^{n})+\frac{d^{2}}{dt^{2}}[F(t,w)](\theta_{2}^{n}),V\right),\text{\,\,}\forall V\in[H^{1}(\Omega)]^{2},
     \end{equation}
      \begin{equation*}
     \left(e_{h}^{n+1}-e_{h}^{n+\frac{1}{2}},V\right)+\frac{\sigma}{4}\left(\alpha.*\Theta(e_{h}^{n+1}+e_{h}^{n+\frac{1}{2}}),\Theta V\right)_{\#}= \frac{\sigma}{4}\left((F(t_{n+1},w_{h}^{n+1})-F(t_{n+1},w^{n+1}))+\right.
     \end{equation*}
   \begin{equation}\label{49}
   \left.(F(t_{n+\frac{1}{2}},w_{h}^{n+\frac{1}{2}})-F(t_{n+\frac{1}{2}},w^{n+\frac{1}{2}})),V\right)+\frac{\sigma^{3}}{96}\left(\alpha.*\Lambda w_{2t}(\theta_{3}^{n+1})+\frac{d^{2}}{dt^{2}}[F(t,w)](\theta_{4}^{n+1}),V\right),\text{\,\,}\forall V\in[H^{1}(\Omega)]^{2}.
     \end{equation}

   Taking $z=e_{h}^{n}$ in equation $(\ref{48})$ and $z=e_{h}^{n+1}$ in equation $(\ref{49})$, simple computations provide
     \begin{equation*}
     \frac{1}{2}\left(\|e_{h}^{n+\frac{1}{2}}\|^{2}-\|e_{h}^{n}\|^{2}-\|e_{h}^{n+\frac{1}{2}}-e_{h}^{n}\|^{2}\right)+\frac{\sigma}{4}
     \left[2\alpha_{1}\|e_{uh}^{n}\|_{\nabla}^{2}+\frac{\alpha_{1}}{2}(\|e_{uh}^{n}-e_{uh}^{n-\frac{1}{2}}\|_{\nabla}^{2}+
     \|e_{uh}^{n}\|_{\nabla}^{2}-\|e_{uh}^{n-\frac{1}{2}}\|_{\nabla}^{2})+\right.
    \end{equation*}
    \begin{equation*}
     \left.2\alpha_{2}\|e_{vh}^{n}\|_{\nabla}^{2}+\frac{\alpha_{2}}{2}(\|e_{vh}^{n}-e_{vh}^{n-\frac{1}{2}}\|_{\nabla}^{2}+
     \|e_{vh}^{n}\|_{\nabla}^{2}-\|e_{vh}^{n-\frac{1}{2}}\|_{\nabla}^{2})\right]=\frac{\sigma}{4}\left(3(F(t_{n},w_{h}^{n})-F(t_{n},w^{n}))-\right.
     \end{equation*}
   \begin{equation}\label{50}
   \left.(F(t_{n-\frac{1}{2}},w_{h}^{n-\frac{1}{2}})-F(t_{n-\frac{1}{2}},w^{n-\frac{1}{2}})),e_{h}^{n}\right)-\frac{5\sigma^{3}}{96}\left(\alpha.*\Lambda w_{2t}(\theta_{1}^{n})+\frac{d^{2}}{dt^{2}}[F(t,w)](\theta_{2}^{n}),e_{h}^{n}\right),
     \end{equation}
     \begin{equation*}
     \frac{1}{2}\left(\|e_{h}^{n+1}\|^{2}-\|e_{h}^{n+\frac{1}{2}}\|^{2}+\|e_{h}^{n+1}-e_{h}^{n+\frac{1}{2}}\|^{2}\right)+\frac{\sigma}{8}
     \left[\alpha_{1}(\|e_{uh}^{n+1}+e_{uh}^{n+\frac{1}{2}}\|_{\nabla}^{2}+\|e_{uh}^{n+1}\|_{\nabla}^{2}-\|e_{uh}^{n+\frac{1}{2}}\|_{\nabla}^{2})+\right.
    \end{equation*}
    \begin{equation*}
     \left.\alpha_{2}(\|e_{vh}^{n+1}+e_{vh}^{n+\frac{1}{2}}\|_{\nabla}^{2}+\|e_{vh}^{n+1}\|_{\nabla}^{2}-\|e_{vh}^{n+\frac{1}{2}}\|_{\nabla}^{2})\right]=
     \frac{\sigma}{4}\left((F(t_{n+1},w_{h}^{n+1})-F(t_{n+1},w^{n+1}))+\right.
     \end{equation*}
   \begin{equation}\label{51}
   \left.(F(t_{n+\frac{1}{2}},w_{h}^{n+\frac{1}{2}})-F(t_{n+\frac{1}{2}},w^{n+\frac{1}{2}})),V\right)+\frac{\sigma^{3}}{96}\left(\alpha.*\Lambda w_{2t}(\theta_{3}^{n+1})+\frac{d^{2}}{dt^{2}}[F(t,w)](\theta_{4}^{n+1}),e_{h}^{n+1}\right).
     \end{equation}

     Both equations hold for every nonnegative $s\leq n$. Since
     \begin{equation}\label{52}
    \underset{s=0}{\overset{n}\sum}\|e_{h}^{s+1}-e_{h}^{s+\frac{1}{2}}\|^{2}-\underset{s=\frac{1}{2}}{\overset{n}\sum}\|e_{h}^{s+\frac{1}{2}}-e_{h}^{s}\|^{2}=
    \underset{s=\frac{1}{2}}{\overset{n+\frac{1}{2}}\sum}\|e_{h}^{s+\frac{1}{2}}-e_{h}^{s}\|^{2}-\underset{s=\frac{1}{2}}{\overset{n}\sum}\|e_{h}^{s+\frac{1}{2}}-e_{h}^{s}\|^{2}=
    \|e_{h}^{n+1}-e_{h}^{n+\frac{1}{2}}\|^{2},
     \end{equation}
     where the summation index $s$ varies with a step size equals $\frac{1}{2}$, summing up equation $(\ref{50})$ from $s=\frac{1}{2},1,...,n$, equation $(\ref{51})$ from $s=0,\frac{1}{2},1,\frac{3}{2},...,n$, adding the resulting equations, utilizing $(\ref{52})$, and rearranging terms, to get
     \begin{equation*}
     \|e_{h}^{n+1}\|^{2}+\|e_{h}^{n+\frac{1}{2}}\|^{2}+\|e_{h}^{n+1}-e_{h}^{n+\frac{1}{2}}\|^{2}+\frac{\sigma}{4}
     \left\{\alpha_{1}\left[\underset{s=\frac{1}{2}}{\overset{n}\sum}(\|e_{uh}^{s+1}+e_{uh}^{s+\frac{1}{2}}\|_{\nabla}^{2}+
     \|e_{uh}^{s}-e_{uh}^{s-\frac{1}{2}}\|_{\nabla}^{2}+4\|e_{uh}^{s}\|_{\nabla}^{2})+\right.\right.
    \end{equation*}
    \begin{equation*}
     \left.\|e_{uh}^{n+1}\|_{\nabla}^{2}+\|e_{uh}^{n}\|_{\nabla}^{2}-\|e_{uh}^{\frac{1}{2}}\|_{\nabla}^{2}-\|e_{uh}^{0}\|_{\nabla}^{2}+
     \|e_{uh}^{1}+e_{uh}^{\frac{1}{2}}\|_{\nabla}^{2}\right]+\alpha_{2}\left[\underset{s=\frac{1}{2}}{\overset{n}\sum}(\|e_{vh}^{s+1}+e_{vh}^{s+\frac{1}{2}}\|_{\nabla}^{2}+
     \right.
     \end{equation*}
    \begin{equation*}
    \left.\left.\|e_{vh}^{s}-e_{vh}^{s-\frac{1}{2}}\|_{\nabla}^{2}+4\|e_{vh}^{s}\|_{\nabla}^{2})+\|e_{vh}^{n+1}\|_{\nabla}^{2}+\|e_{vh}^{n}\|_{\nabla}^{2}-\|e_{vh}^{\frac{1}{2}}\|_{\nabla}^{2}-\|e_{vh}^{0}\|_{\nabla}^{2}+
     \|e_{vh}^{1}+e_{vh}^{\frac{1}{2}}\|_{\nabla}^{2}\right]\right\}=
     \end{equation*}
      \begin{equation*}
     2\|e_{h}^{\frac{1}{2}}\|^{2}+\frac{\sigma}{4}\underset{s=\frac{1}{2}}{\overset{n}\sum}\left(3(F(t_{s},w_{h}^{s})-F(t_{s},w^{s}))-
   (F(t_{s-\frac{1}{2}},w_{h}^{s-\frac{1}{2}})-F(t_{s-\frac{1}{2}},w^{s-\frac{1}{2}})),e_{h}^{s}\right)+
    \end{equation*}
    \begin{equation*}
    \frac{\sigma}{4}\underset{s=0}{\overset{n}\sum}\left((F(t_{s+1},w_{h}^{s+1})-F(t_{s+1},w^{s+1}))+
   (F(t_{s+\frac{1}{2}},w_{h}^{s+\frac{1}{2}})-F(t_{s+\frac{1}{2}},w^{s+\frac{1}{2}})),e_{h}^{s+1}\right)-
    \end{equation*}
   \begin{equation*}
   \frac{\sigma^{3}}{96}\left[5\underset{s=\frac{1}{2}}{\overset{n}\sum}\left(\alpha.*\Lambda w_{2t}(\theta_{1}^{s})+\frac{d^{2}}{dt^{2}}[F(t,w)](\theta_{2}^{s}), e_{h}^{s}\right)-\underset{s=0}{\overset{n}\sum}\left(\alpha.*\Lambda w_{2t}(\theta_{3}^{s+1})+\frac{d^{2}}{dt^{2}}[F(t,w)](\theta_{4}^{s+1}), e_{h}^{s+1}\right)\right].
     \end{equation*}

     This equation implies
     \begin{equation*}
     \|e_{h}^{n+1}\|^{2}+\|e_{h}^{n+\frac{1}{2}}\|^{2}+\frac{\sigma}{4}\underset{1\leq l\leq2}{\min}\{\alpha_{l}\}
     \left[\underset{s=\frac{1}{2}}{\overset{n}\sum}(\|\Theta(e_{h}^{s+1}+e_{h}^{s+\frac{1}{2}})\|_{\#}^{2}+
     \|\Theta(e_{h}^{s}-e_{h}^{s-\frac{1}{2}})\|_{\#}^{2}+3\|\Theta e_{h}^{s}\|_{\#}^{2})+\right.
    \end{equation*}
    \begin{equation*}
     \left.\|\Theta e_{h}^{n+1}\|_{\#}^{2}+\|\Theta e_{h}^{n}\|_{\#}^{2}\right] \leq 2\|e_{h}^{\frac{1}{2}}\|^{2}+\frac{\sigma}{4}\underset{1\leq l\leq2}{\max}\{\alpha_{l}\}\|\Theta e_{h}^{0}\|_{\#}^{2}+\frac{\sigma}{4}\underset{s=\frac{1}{2}}{\overset{n}\sum}\left(3(F(t_{s},w_{h}^{s})-F(t_{s},w^{s}))-\right.
     \end{equation*}
      \begin{equation*}
      \left.(F(t_{s-\frac{1}{2}},w_{h}^{s-\frac{1}{2}})-F(t_{s-\frac{1}{2}},w^{s-\frac{1}{2}})),e_{h}^{s}\right)+
      \frac{\sigma}{4}\underset{s=0}{\overset{n}\sum}\left((F(t_{s+1},w_{h}^{s+1})-F(t_{s+1},w^{s+1}))+\right.
    \end{equation*}
    \begin{equation*}
    \left.(F(t_{s+\frac{1}{2}},w_{h}^{s+\frac{1}{2}})-F(t_{s+\frac{1}{2}},w^{s+\frac{1}{2}})),e_{h}^{s+1}\right)-
    \frac{\sigma^{3}}{96}\left[5\underset{s=\frac{1}{2}}{\overset{n}\sum}\left(\alpha.*\Lambda w_{2t}(\theta_{1}^{s})+\frac{d^{2}}{dt^{2}}[F(t,w)](\theta_{2}^{s}), e_{h}^{s}\right)\right.
    \end{equation*}
   \begin{equation}\label{53}
   \left.-\underset{s=0}{\overset{n}\sum}\left(\alpha.*\Lambda w_{2t}(\theta_{3}^{s+1})+\frac{d^{2}}{dt^{2}}[F(t,w)](\theta_{4}^{s+1}), e_{h}^{s+1}\right)\right].
     \end{equation}

   Applying the Cauchy-Schwarz inequality along with estimate $(\ref{39})$ in Lemma $\ref{l2}$ and using the inner product $\left(.,.\right)$, straightforward computations result in
  \begin{equation}\label{54}
   \left(3(F(t_{s},w_{h}^{s})-F(t_{s},w^{s}))-(F(t_{s-\frac{1}{2}},w_{h}^{s-\frac{1}{2}})-F(t_{s-\frac{1}{2}},w^{s-\frac{1}{2}})),e^{s}\right)\leq (\frac{1}{2}+18C_{F}^{2})(\|e_{h}^{s}\|^{2}+\|e_{h}^{s-\frac{1}{2}}\|^{2}),
     \end{equation}
    \begin{equation}\label{55}
   \left(F(t_{s+1},w_{h}^{s+1})-F(t_{s+1},w^{s+1})+F(t_{s+\frac{1}{2}},w_{h}^{s+\frac{1}{2}})-F(t_{s+\frac{1}{2}},w^{s+\frac{1}{2}}),e^{s+1}\right)\leq (\frac{1}{2}+2C_{F}^{2})(\|e_{h}^{s+1}\|^{2}+\|e_{h}^{s+\frac{1}{2}}\|^{2}),
     \end{equation}
    \begin{equation}\label{56}
   -5\sigma^{3}\left(\alpha.*\Lambda w_{2t}(\theta_{1}^{s})+\frac{d^{2}}{dt^{2}}[F(t,w)](\theta_{2}^{s}), e_{h}^{s}\right)\leq
   \frac{\sigma}{2}\|e_{h}^{s}\|^{2}+\frac{25\sigma^{5}}{2}\||\alpha.*\Lambda w_{2t}+\frac{d^{2}}{dt^{2}}[F(t,w)]|\|_{\infty}^{2},
     \end{equation}
   \begin{equation}\label{57}
   \sigma^{3}\left(\alpha.*\Lambda w_{2t}(\theta_{3}^{s+1})+\frac{d^{2}}{dt^{2}}[F(t,w)](\theta_{4}^{s+1}), e_{h}^{s+1}\right)\leq
   \frac{\sigma}{2}\|e_{h}^{s+1}\|^{2}+\frac{\sigma^{5}}{2}\||\alpha.*\Lambda w_{2t}+\frac{d^{2}}{dt^{2}}[F(t,w)]|\|_{\infty}^{2},
     \end{equation}
  where $C_{F}$ is the positive constant given by equation $(\ref{41b})$. Substituting estimates $(\ref{54})$-$(\ref{57})$, into inequality $(\ref{53})$ and rearranging terms, it is easy to see that
     \begin{equation*}
     \|e_{h}^{n+1}\|^{2}+\|e_{h}^{n+\frac{1}{2}}\|^{2}+\frac{\sigma}{4}\underset{1\leq l\leq2}{\min}\{\alpha_{l}\}
     \left[\underset{s=\frac{1}{2}}{\overset{n}\sum}(\|\Theta(e_{h}^{s+1}+e_{h}^{s+\frac{1}{2}})\|_{\#}^{2}+
     \|\Theta(e_{h}^{s}-e_{h}^{s-\frac{1}{2}})\|_{\#}^{2}+3\|\Theta e_{h}^{s}\|_{\#}^{2})+\right.
    \end{equation*}
    \begin{equation*}
     \left.\|\Theta e_{h}^{n+1}\|_{\#}^{2}+\|\Theta e_{h}^{n}\|_{\#}^{2}\right] \leq 2\|e_{h}^{\frac{1}{2}}\|^{2}+\frac{\sigma}{4}\left(\underset{1\leq l\leq2}{\max}\{\alpha_{l}\}\|\Theta e_{h}^{0}\|_{\#}^{2}+(\frac{1}{2}+18C_{F}^{2})\|e_{h}^{0}\|^{2}\right)+
    \end{equation*}
    \begin{equation*}
    \sigma(1+20C_{F}^{2})\underset{s=0}{\overset{n}\sum}(\|e_{h}^{s+1}\|^{2}+\|e_{h}^{s+\frac{1}{2}}\|^{2})+
    13\sigma^{5}(2n+1)\||\alpha.*\Lambda w_{2t}+\frac{d^{2}}{dt^{2}}[F(t,w)]|\|_{\infty}^{2}.
     \end{equation*}

   For small values of the time step $\sigma$, the application of the Gronwall inequality yields
    \begin{equation*}
     \|e_{h}^{n+1}\|^{2}+\|e_{h}^{n+\frac{1}{2}}\|^{2}+\frac{\sigma}{4}\underset{1\leq l\leq2}{\min}\{\alpha_{l}\}
     \left[\underset{s=\frac{1}{2}}{\overset{n}\sum}(\|\Theta(e_{h}^{s+1}+e_{h}^{s+\frac{1}{2}})\|_{\#}^{2}+
     \|\Theta(e_{h}^{s}-e_{h}^{s-\frac{1}{2}})\|_{\#}^{2}+3\|\Theta e_{h}^{s}\|_{\#}^{2})+\right.
    \end{equation*}
    \begin{equation*}
     \left.\|\Theta e_{h}^{n+1}\|_{\#}^{2}+\|\Theta e_{h}^{n}\|_{\#}^{2}\right] \leq \left[2\|e_{h}^{\frac{1}{2}}\|^{2}+\frac{\sigma}{4}\left(\underset{1\leq l\leq2}{\max}\{\alpha_{l}\}\|\Theta e_{h}^{0}\|_{\#}^{2}+(\frac{1}{2}+18C_{F}^{2})\|e_{h}^{0}\|^{2}\right)+\right.
    \end{equation*}
     \begin{equation}\label{58}
     \left.13\sigma^{5}(2n+1)\||\alpha.*\Lambda w_{2t}+\frac{d^{2}}{dt^{2}}[F(t,w)]|\|_{\infty}^{2}\right]\exp\left((2n+1)\sigma(1+20C_{F}^{2})\right).
     \end{equation}

    Utilizing the initial conditions $(\ref{s5})$, Lemma $\ref{l1}$, equations $(\ref{33a})$-$(\ref{33b})$, and the definition of the norms $\|.\|$ and $\|.\|_{\nabla}$, given by equation $(\ref{12})$, it is not hard to show that
        \begin{equation*}
   \|e_{h}^{\frac{1}{2}}\|^{2}=\|w_{h}^{\frac{1}{2}}-w^{\frac{1}{2}}\|^{2}\leq 2(\|w_{h}^{\frac{1}{2}}-\widetilde{w}^{\frac{1}{2}}\|^{2}+\|\widetilde{w}^{\frac{1}{2}}-w^{\frac{1}{2}}\|^{2})\leq 2(\|u_{h}^{\frac{1}{2}}-\widetilde{u}^{\frac{1}{2}}\|^{2}+\|v_{h}^{\frac{1}{2}}-\widetilde{v}^{\frac{1}{2}}\|^{2}+\frac{\sigma^{4}}{64}\|w_{2t}(\theta_{0})\|^{2})\leq
    \end{equation*}
     \begin{equation*}
  2[\|p_{h}(u_{0}+\frac{\sigma}{2}u_{1})-(u_{0}+\frac{\sigma}{2}u_{1})\|^{2}+\|p_{h}(v_{0}+\frac{\sigma}{2}v_{1})-(v_{0}+\frac{\sigma}{2}v_{1})\|^{2}+
  \frac{\sigma^{4}}{64}\|w_{2t}(\theta_{0})\|^{2}]\leq 2[\frac{\sigma^{4}}{64}\|w_{2t}(\theta_{0})\|^{2}+
     \end{equation*}
        \begin{equation}\label{59}
    C_{1}h^{2q+2}(\|u_{0}+\frac{\sigma}{2}u_{1}\|^{2}_{q+1}+\|v_{0}+\frac{\sigma}{2}v_{1}\|^{2}_{q+1})]\leq
  2(\sqrt{C_{1}}(\|u_{0}+\frac{\sigma}{2}u_{1}\|_{q+1}+\|v_{0}+\frac{\sigma}{2}v_{1}\|_{q+1})+\frac{1}{8}\|w_{2t}(\theta_{0})\|)^{2}(\sigma^{2}+h^{q+1})^{2},
     \end{equation}
     \begin{equation}\label{60}
   \|e_{h}^{0}\|^{2}=\|u_{h}^{0}-u_{0}\|^{2}+\|v_{h}^{0}-v_{0}\|^{2}=\|p_{h}u_{0}-u_{0}\|^{2}+\|p_{h}v_{0}-v_{0}\|^{2}\leq C_{1}h^{2q+2}(\|u_{0}\|_{q+1}+\|v_{0}\|_{q+1})^{2},
    \end{equation}
    \begin{equation}\label{61}
   \|\Theta e_{h}^{0}\|_{\#}^{2}=\|u_{h}^{0}-u_{0}\|_{\nabla}^{2}+\|v_{h}^{0}-v_{0}\|_{\nabla}^{2}=\|p_{h}u_{0}-u_{0}\|_{\nabla}^{2}+\|p_{h}v_{0}-v_{0}\|_{\nabla}^{2}\leq C_{1}h^{2q}(\|u_{0}\|_{q+1}+\|v_{0}\|_{q+1})^{2},
   \end{equation}

   Because $n\leq N=\frac{T}{\sigma}$, then $(2n+1)\sigma\leq 2T+\sigma\leq 3T$. Substituting estimates $(\ref{59})$-$(\ref{61})$, into estimate $(\ref{58})$ and rearranging terms, to obtain
    \begin{equation*}
     \|e_{h}^{n+1}\|^{2}+\|e_{h}^{n+\frac{1}{2}}\|^{2}+\frac{\sigma}{4}\underset{1\leq l\leq2}{\min}\{\alpha_{l}\}
     \left[\underset{s=\frac{1}{2}}{\overset{n}\sum}(\|\Theta(e_{h}^{s+1}+e_{h}^{s+\frac{1}{2}})\|_{\#}^{2}+
     \|\Theta(e_{h}^{s}-e_{h}^{s-\frac{1}{2}})\|_{\#}^{2}+3\|\Theta e_{h}^{s}\|_{\#}^{2})+\|\Theta e_{h}^{n+1}\|_{\#}^{2}+\right.
    \end{equation*}
    \begin{equation*}
     \left.\|\Theta e_{h}^{n}\|_{\#}^{2}\right] \leq \exp\left(3T(1+20C_{F}^{2})\right)\left[\frac{\sqrt{C_{1}\sigma}}{2}\left(\underset{1\leq l\leq2}{\max}\{\alpha_{l}\}+
     (\frac{1}{2}+18C_{F}^{2})h^{2}\right)^{\frac{1}{2}}(\|u_{0}\|_{q+1}+\|v_{0}\|_{q+1})+\right.
    \end{equation*}
     \begin{equation*}
     \left.2\sqrt{C_{1}}(\|u_{0}+\frac{\sigma}{2}u_{1}\|_{q+1}+\|v_{0}+\frac{\sigma}{2}v_{1}\|_{q+1}+\frac{1}{8}\||w_{2t}|\|_{\infty})+\sqrt{39T}\||\alpha.*\Lambda w_{2t}+\frac{d^{2}}{dt^{2}}[F(t,w)]|\|_{\infty}^{2}\right]^{2}(\sigma^{2}+h^{q})^{2}.
     \end{equation*}

   Since $\|e_{h}^{n+1}\|^{2}+\|e_{h}^{n+\frac{1}{2}}\|^{2}=\|u_{h}^{n+1}-u^{n+1}\|^{2}+\|u_{h}^{n+\frac{1}{2}}-u^{n+\frac{1}{2}}\|^{2}+\|v_{h}^{n+1}-v^{n+1}\|^{2}+
   \|v_{h}^{n+\frac{1}{2}}-v^{n+\frac{1}{2}}\|^{2}$, this completes the proof of Theorem $\ref{t2}$.
   \end{proof}

   \section{Numerical experiments}\label{sec4}
    In this section, we perform some numerical examples to verify the theoretical results provided in Theorems $\ref{t1}$-$\ref{t2}$, in an approximate solution of the two-dimensional Gray-Scott system $(\ref{1})$ with initial-boundary conditions $(\ref{2})$-$(\ref{3})$, and to demonstrate the efficiency of the developed computational approach $(\ref{s1})$-$(\ref{s6})$. We set $q=3$ and suppose that the time step $\sigma\in\{2^{-l},\text{\,\,}l=5,6,...,9\}$ while the grid mesh $h=2^{-l}$, for $l=3,4,...,7$. The  $L^{\infty}(0,T;[L^{2}(\Omega)]^{2})$-norm is used to calculate: exact solution $(u,v)$, computed solution $(u_{h},v_{h})$, and error $(e_{uh},e_{vh})$. That is,
     \begin{equation*}
       \||z|\|_{\infty}=\underset{0\leq n\leq N}{\max}\|z^{n}\|,
       \end{equation*}
       where $z\in \{u,u_{h},v,v_{h},e_{uh},e_{vh},\}$. The convergence order $CO(h,\sigma)$ of the constructed strategy is estimated using the formula
       \begin{equation*}
        CO(h,\sigma)=\log_{2}\left(\frac{\||e_{z2h,\sigma}|\|_{\infty}}{\||e_{zh,\sigma}|\|_{\infty}}\right),\text{\,\,\,\,and\,\,\,\,}CO(h,\sigma)=
        \log_{2}\left(\frac{\||e_{zh,2\sigma}|\|_{\infty}}{\||e_{zh,\sigma}|\|_{\infty}}\right).
       \end{equation*}
        Finally, the numerical calculations are carried out with the help of MATLAB R$2007b$.\\

      $\bullet$ \textbf{Example $1$}. Let $\Omega=[-1,1]\times[-1,1]$, be the fluid region and let $T=1$, be the final time. We consider the two-dimensional Gray-Scott system defined in \cite{zwz} as
     \begin{equation*}
      \left\{
        \begin{array}{ll}
          u_{t}-\alpha_{1}\Delta u=\beta_{0}(1-u)-uv^{2}, & \hbox{on $\Omega\times[0,\text{\,}1]$} \\
          \text{\,}\\
          v_{t}-\alpha_{2}\Delta v=-(\beta_{0}+k_{0})v+uv^{2}, & \hbox{on $\Omega\times[0,\text{\,}1]$}
        \end{array}
      \right.
     \end{equation*}
     where $\alpha_{1}=\alpha_{2}=\beta_{0}=1$, and $k_{0}=0$. The analytical solution $(u,v)$ is defined as
      \begin{equation*}
     u(x,y,t)=\cos(\pi x)\cos(\pi y)\sin(t),\text{\,\,\,\,\,}v(x,y,t)=2\cos(\pi x)\cos(\pi y)\sin(t).
     \end{equation*}
    Both initial and homogeneous Neumann boundary conditions are determined from the analytical solution.\\

       \textbf{Table 1} $\label{T1}$. Stability and convergence order $CO(h,\sigma)$ of the developed two-stage explicit/implicit approach with spectral orthogonal basis Galerkin finite element scheme with varying space step $h$ and time step $\sigma$.
          \begin{equation*}
          \begin{array}{c}
          \text{\,proposed computational approach:\,\,}\sigma=2^{-8}\\
           \begin{tabular}{cccccccccc}
            \hline
         $h$  &  $\||u|\|_{\infty}$&  $\||u_{h}|\|_{\infty}$&  $\||e_{u}(h,\sigma)|\|_{\infty}$ & $CO(h,\sigma)$ &  $\||v|\|_{\infty}$ &  $\||v_{h}|\|_{\infty}$ & $\||e_{v}(h,\sigma)|\|_{\infty}$ & $CO(h,\sigma)$ & CPU (s) \\
             \hline
            $2^{-3}$ & $0.5316$ & $0.5314$ & $6.1179\times10^{-4}$ & ---      & $1.0632$ & $1.0631$ & $3.2705\times10^{-4}$ &  ---    & $3.4681 $\\
            $2^{-4}$ & $0.5316$ & $0.5314$ & $8.1838\times10^{-5}$ & $2.9022$ & $1.0632$ & $1.0632$ & $4.0593\times10^{-5}$ &  3.0102 & $6.9317 $\\
            $2^{-5}$ & $0.5316$ & $0.5315$ & $1.0322\times10^{-5}$ & $2.9871$ & $1.0633$ & $1.0632$ & $5.3345\times10^{-6}$ &  2.9278 & $ 13.7698 $\\
            $2^{-6}$ & $0.5317$ & $0.5316$ & $1.3234\times10^{-6}$ & $2.9634$ & $1.0633$ & $1.0633$ & $6.6922\times10^{-7}$ &  2.9948 & $ 26.0938 $\\
            $2^{-7}$ & $0.5320$ & $0.5318$ & $1.6521\times10^{-7}$ & $3.0019$ & $1.0634$ & $1.0633$ & $8.3653\times10^{-8}$ &  3.0000 & $ 55.9712$\\
           \hline
          \end{tabular}
           \end{array}
            \end{equation*}
          \begin{equation*}
          \begin{array}{c}
          \text{\,proposed computational approach:\,\,}h=2^{-6}\\
         \begin{tabular}{cccccccccc}
            \hline
         $\sigma$  &  $\||u|\|_{\infty}$&  $\||u_{h}|\|_{\infty}$&  $\||e_{u}(h,\sigma)|\|_{\infty}$ & $CO(h,\sigma)$ &  $\||v|\|_{\infty}$ &  $\||v_{h}|\|_{\infty}$ & $\||e_{v}(h,\sigma)|\|_{\infty}$ & $CO(h,\sigma)$ & CPU (s) \\
             \hline
            $2^{-5}$ & $0.5314$ & $0.5311$ & $7.7470\times10^{-5}$ & ---    & $1.0631$ & $1.0628$ & $4.1198\times10^{-5}$    &   -- & $3.4681 $\\
            $2^{-6}$ & $0.5314$ & $0.5312$ & $1.9354\times10^{-5}$ & 2.0001 & $1.0631$ & $1.0630$ & $1.0354\times10^{-6}$ & 1.9924  & $6.9317 $\\
            $2^{-7}$ & $0.5315$ & $0.5314$ & $4.8443\times10^{-6}$ & 1.9983 & $1.0632$ & $1.0630$ & $2.6324\times10^{-7}$ & 1.9758  & $13.7698 $\\
            $2^{-8}$ & $0.5317$ & $0.5316$ & $1.2027\times10^{-6}$ & 2.0100 & $1.0633$ & $1.0633$ & $6.6922\times10^{-8}$ & 2.0000 & $26.0938 $\\
            $2^{-9}$ & $0.5319$ & $0.5317$ & $2.7605\times10^{-7}$ & 2.1233 & $1.0636$ & $1.0635$ & $1.6771\times10^{-8}$ & 1.9965 & $55.9712 $\\
           \hline
          \end{tabular}
          \end{array}
           \end{equation*}
          \text{\,}\\
          \text{\,}\\
          $\bullet$ \textbf{Example $2$}. Suppose that $\Omega=[0,\text{\,}1]^{2}$, is the fluid region and $T=10$ is the time interval. Consider the following $2D$ Gray-Scott  model described in \cite{zwz} as
       \begin{equation*}
      \left\{
        \begin{array}{ll}
          u_{t}-\alpha_{1}\Delta u=\beta_{0}(1-u)-uv^{2}, & \hbox{on $\Omega\times[0,\text{\,}1]$}\\
          \text{\,}\\
          v_{t}-\alpha_{2}\Delta v=-(\beta_{0}+k_{0})v+uv^{2}, & \hbox{on $\Omega\times[0,\text{\,}1]$}
        \end{array}
      \right.
     \end{equation*}
     where $\alpha_{1}=1.6\times10^{-5}$ and $\alpha_{2}=8\times10^{-6}$.\\

      In general, the Gray-Scott problem $(\ref{1})$-$(\ref{3})$ does not possess a known closed form analytical solution. To test proposed numerical methods, some authors \cite{45tir,tir} assume that the exact solution is a high-fidelity approximate solution computed using a very small time step. On the other hand, the Gray-Scott model is known for its slow dynamics: patterns only appear gradually and transients may look very similar at earlier times (for example, when the time is greater than $100$), whereas the model usually settles into its characteristic long-term structures only at much later times (for instance, time is greater than $2000$). An alternate computational technique is motivated by the numerical difficulties of calculating such high-fidelity solutions up to time greater than $2000$ or beyond. Additionally, to construct exact solutions, we add source terms to the Gray-Scott model and prescribe smooth fake solutions $\bar{u}$ and $\bar{v}$. As a result, we consider the following modified system
      \begin{equation}\label{me}
      \left\{
        \begin{array}{ll}
          u_{t}-\alpha_{1}\Delta u=\beta_{0}(1-u)-uv^{2}+f_{1}(u), & \hbox{on $\Omega\times[0,\text{\,}10]$} \\
          \text{\,}\\
          v_{t}-\alpha_{2}\Delta v=-(\beta_{0}+k_{0})v+uv^{2}+f_{2}(v), & \hbox{on $\Omega\times[0,\text{\,}10]$}
        \end{array}
      \right.
     \end{equation}
     where $\beta_{0}=3.7\times10^{-2}$ and $k_{0}=6\times10^{-2}$. The source terms $f_{1}(u)$ and $f_{2}(v)$ are chosen so that the system of equations $(\ref{me})$ has an exact solution $(\bar{u},\bar{v})$ defined as
      \begin{equation*}
     \bar{u}(x,y,t)=1-b\cos(2\pi x)\cos(2\pi y)\sin(2\pi t),\text{\,\,\,\,\,}\bar{v}(x,y,t)=\frac{1}{4}[1+\cos(2\pi x)\cos(2\pi y)\sin(2\pi t)],
     \end{equation*}
      where $b\in(0,\text{\,}1)$. Without loss of generality, we set $b=\frac{1}{2}$. The initial and homogeneous Neumann boundary conditions are obtained from the analytical solution $(\bar{u},\bar{v})$.\\

     \textbf{Table 2} $\label{T2}$. Stability and convergence order $CO(h,\sigma)$ of the developed two-stage explicit/implicit approach with spectral orthogonal basis Galerkin finite element scheme with varying space step $h$ and time step $\sigma$.
          \begin{equation*}
          \begin{array}{c}
          \text{\,proposed computational approach:\,\,}\sigma=2^{-8}\\
           \begin{tabular}{cccccccccc}
            \hline
         $h$  &  $\||u|\|_{\infty}$&  $\||u_{h}|\|_{\infty}$&  $\||e_{u}(h,\sigma)|\|_{\infty}$ & $CO(h,\sigma)$ &  $\||v|\|_{\infty}$ &  $\||v_{h}|\|_{\infty}$ & $\||e_{v}(h,\sigma)|\|_{\infty}$ & $CO(h,\sigma)$ & CPU (s) \\
             \hline
            $2^{-3}$ & $3.2658$ & $3.2657$ & $3.2176\times10^{-3}$ & ---      & $0.8535$ & $0.8533$ & $1.7509\times10^{-3}$ &  ---    & $3.9086 $\\
            $2^{-4}$ & $3.2658$ & $3.2659$ & $4.0991\times10^{-4}$ & $2.9726$ & $0.8537$ & $0.8535$ & $2.1860\times10^{-4}$ &  3.0017 & $8.4192 $\\
            $2^{-5}$ & $3.2659$ & $3.2660$ & $5.0365\times10^{-5}$ & $3.0248$ & $0.8537$ & $0.8536$ & $2.9300\times10^{-5}$ &  2.8993 & $ 16.7517 $\\
            $2^{-6}$ & $3.2661$ & $3.2660$ & $5.8197\times10^{-6}$ & $3.1134$ & $0.8538$ & $0.8536$ & $3.7183\times10^{-6}$ &  2.9782 & $ 31.7813 $\\
            $2^{-7}$ & $3.2663$ & $3.2662$ & $6.4866\times10^{-7}$ & $3.1654$ & $0.8540$ & $0.8538$ & $4.5120\times10^{-7}$ &  3.0428 & $ 63.5213$\\
           \hline
          \end{tabular}
           \end{array}
            \end{equation*}
          \begin{equation*}
          \begin{array}{c}
          \text{\,proposed computational approach:\,\,}h=2^{-6}\\
         \begin{tabular}{cccccccccc}
            \hline
         $\sigma$  &  $\||u|\|_{\infty}$&  $\||u_{h}|\|_{\infty}$&  $\||e_{u}(h,\sigma)|\|_{\infty}$ & $CO(h,\sigma)$ &  $\||v|\|_{\infty}$ &  $\||v_{h}|\|_{\infty}$ & $\||e_{v}(h,\sigma)|\|_{\infty}$ & $CO(h,\sigma)$ & CPU (s) \\
             \hline
            $2^{-5}$ & $3.2656$ & $3.2654$ & $4.9711\times10^{-4}$ & ---    & $0.8533$ & $0.8534$ & $2.3365\times10^{-4}$    &   -- & $3.9086 $\\
            $2^{-6}$ & $3.2657$ & $3.2655$ & $1.2448\times10^{-4}$ & 1.9976 & $0.8535$ & $0.8534$ & $5.9109\times10^{-5}$ & 1.9829  & $8.4192 $\\
            $2^{-7}$ & $3.2658$ & $3.2657$ & $3.1369\times10^{-5}$ & 1.9885 & $0.8537$ & $0.8536$ & $1.4742\times10^{-5}$ & 2.0034  & $16.7517 $\\
            $2^{-8}$ & $3.2659$ & $3.2658$ & $7.8423\times10^{-6}$ & 2.0000 & $0.8538$ & $0.8536$ & $3.5041\times10^{-6}$ & 2.0728  & $31.7813 $\\
            $2^{-9}$ & $3.2661$ & $3.2660$ & $1.9539\times10^{-6}$ & 2.0049 & $0.8539$ & $0.8538$ & $8.0935\times10^{-7}$ & 2.1142  & $63.5213 $\\
           \hline
          \end{tabular}
          \end{array}
           \end{equation*}
        \text{\,}\\
          \text{\,}\\
        \text{\,}\\
      $\bullet$ \textbf{Example $3$} \cite{45tir}. Suppose that $\Omega=[0,\text{\,}2.5]^{2}$, is the fluid region and let $T=1000$, be the final time. We analyze the constructed two-stage explicit/implicit approach $(\ref{s1})$-$(\ref{s6})$. We consider the two-dimensional Gray-Scott problem defined by
      \begin{equation*}
      \left\{
        \begin{array}{ll}
          u_{t}-\alpha_{1}\Delta u=\beta_{0}(1-u)-uv^{2}, & \hbox{on $\Omega\times(0,\text{\,}1000]$} \\
          \text{\,}\\
          v_{t}-\alpha_{2}\Delta v=-(\beta_{0}+k_{0})v+uv^{2}, & \hbox{on $\Omega\times(0,\text{\,}1000]$}
        \end{array}
      \right.
     \end{equation*}
     with the initial conditions
      \begin{equation*}
     v_{t}(x,y,0)=v(x,y,0)=\left\{
                \begin{array}{ll}
                  \frac{1}{4}\sin^{2}(4\pi x)\sin^{2}(4\pi y), & \hbox{if $1\leq x,y\leq 1.5$} \\
                  \text{\,}\\
                  0, & \hbox{otherwise}
                \end{array}
              \right.,\text{\,\,\,\,}u_{t}(x,y,0)=u(x,y,0)=1-2v(x,y,0),
     \end{equation*}
      and Neumann boundary conditions. Here: $\alpha_{1}=8\times10^{-5}$, $\alpha_{2}=4\times10^{-5}$, $\beta_{0}=3\times10^{-2}$, and $k_{0}=6\times10^{-2}$. Because of the lack of an analytical solution, to check the theoretical result provided by Theorems $\ref{t1}$-$\ref{t2}$, a high-fidelity approximate solution computed using a very small time step $\sigma=2^{-10}$, is considered as the exact solution.\\

       \textbf{Table 3} $\label{T3}$. Stability and convergence order $CO(h,\sigma)$ of the developed two-stage explicit/implicit approach with spectral orthogonal basis Galerkin finite element scheme with varying space step $h$ and time step $\sigma$.
          \begin{equation*}
          \begin{array}{c}
          \text{\,proposed computational approach:\,\,}\sigma=2^{-8}\\
           \begin{tabular}{cccccccc}
            \hline
         $h$  &    $\||u_{h}|\|_{\infty}$&  $\||e_{u}(h,\sigma)|\|_{\infty}$ & $CO(h,\sigma)$ &  $\||v_{h}|\|_{\infty}$ & $\||e_{v}(h,\sigma)|\|_{\infty}$ & $CO(h,\sigma)$ & CPU (s) \\
             \hline
            $2^{-3}$ &  $7.8700$ & $7.9338\times10^{-2}$ & ---      &  $0.1673$ & $4.8520\times10^{-2}$ &  ---    & $11.3698 $\\
            $2^{-4}$ &  $7.8701$ & $1.3789\times10^{-2}$ & $2.5246$ &  $0.1673$ & $8.6813\times10^{-3}$ &  2.4826 & $24.3951 $\\
            $2^{-5}$ &  $7.8702$  & $2.5733\times10^{-3}$ & $2.4218$&  $0.1674$ & $1.6026\times10^{-3}$ &  2.4375 & $48.6341 $\\
            $2^{-6}$ &  $7.8702$ & $4.2503\times10^{-4}$ & $2.5980$ &  $0.1674$ & $2.6879\times10^{-4}$ &  2.5759 & $96.5678 $\\
            $2^{-7}$ &  $7.8703$  & $6.8121\times10^{-5}$ & $2.6414$&  $0.1675$ & $4.1966\times10^{-5}$ &  2.6792 & $207.5145$\\
           \hline
          \end{tabular}
           \end{array}
            \end{equation*}
          \begin{equation*}
          \begin{array}{c}
          \text{\,proposed computational approach:\,\,}h=2^{-6}\\
         \begin{tabular}{cccccccc}
            \hline
         $\sigma$  &   $\||u_{h}|\|_{\infty}$&  $\||e_{u}(h,\sigma)|\|_{\infty}$ & $CO(h,\sigma)$ &  $\||v_{h}|\|_{\infty}$ & $\||e_{v}(h,\sigma)|\|_{\infty}$ & $CO(h,\sigma)$ & CPU (s) \\
             \hline
            $2^{-5}$ &  $7.8702$ & $9.1221\times10^{-3}$ & ---    &  $0.1672$ & $6.2335\times10^{-3}$ &   --   & $11.3698 $\\
            $2^{-6}$ &  $7.8703$ & $3.2998\times10^{-3}$ & 1.4670 &  $0.1672$ & $2.2696\times10^{-3}$ & 1.4576 & $24.3951 $\\
            $2^{-7}$ &  $7.8703$ & $1.2128\times10^{-3}$ & 1.4444 &  $0.1674$ & $8.0165\times10^{-4}$ & 1.5014 & $48.6341 $\\
            $2^{-8}$ &  $7.8704$ & $4.2503\times10^{-4}$ & 1.5127 &  $0.1674$ & $2.6879\times10^{-4}$ & 1.5765 & $96.5678 $\\
            $2^{-9}$ &  $7.8705$ & $1.3920\times10^{-4}$ & 1.6104 &  $0.1676$ & $8.6681\times10^{-5}$ & 1.6327 & $207.5145 $\\
           \hline
          \end{tabular}
          \end{array}
           \end{equation*}

        \textbf{Tables} $1$-$3$ indicate that the constructed two-stage explicit/implicit approach $(\ref{s1})$-$(\ref{s6})$ is temporal second-order accurate and spatial third-order convergent. In addition, Figures $\ref{figure1}$-$\ref{figure3}$ show that the developed strategy is unconditionally stable. These computational results confirm the theoretical studies provided in Theorems $\ref{t1}$-$\ref{t2}$. Furthermore, the graphs illustrate how patterns will eventually split and combine to create a stable, center-symmetric speckle pattern. As the value $\beta_{0}$ rises, the resulting pattern progressively changes from a simple symmetric point to a symmetric design where curves and lines coexist. At $\beta_{0}=3\times10^{-2}$, a complex pattern emerges (see Figure $\ref{figure3}$). Finally, the analysis discussed in this section indicates that the new numerical approach $(\ref{s1})$-$(\ref{s6})$ for simulating the initial-boundary value problem $(\ref{1})$-$(\ref{3})$, computes efficiently at the first and second stages the numerical solutions and maintains strong stability and high-accuracy.

     \section{General conclusions and future investigations}\label{sec5}
     This paper has constructed a strong second-order two-stage explicit/implicit technique with spectral orthogonal basis Galerkin finite element method for simulating a two-dimensional Gray-Scott system $(\ref{1})$ with initial-boundary conditions $(\ref{2})$-$(\ref{3})$. The developed numerical strategy $(\ref{s1})$-$(\ref{s6})$ approximates in time using a second-order explicit-implicit method whereas the Galerkin finite element formulation along with a spectral orthogonal basis are employed to discretize the space derivatives. The errors increased at the first phase are balanced by the ones decreased at the second phase so that the stability is maintained. Additionally, the use of the spectral orthogonal basis minimizes the space errors. The theory suggests that the proposed algorithm is unconditionally stable, temporal second-order accurate and spatial $qth$-order convergent in the $L^{\infty}(0,T;[L^{\infty}(\Omega)]^{2})$-norm. Numerical experiments are carried out to confirm the theoretical results and to establish the performance of the constructed numerical method. As a result, the new computational technique computes efficiently numerical solutions and preserves a strong stability and high-order accuracy. Our future works will develop a strong second-order two-stage explicit/implicit approach with spectral orthogonal basis Galerkin finite element method for solving the two-dimensional parabolic interface problems.

         \subsection*{Ethical Approval}
          Not applicable.
         \subsection*{Availability of supporting data}
          Not applicable.
         \subsection*{Declaration of Interest Statement}
          The author declares that he has no conflict of interests.
         \subsection*{Authors' contributions}
          The whole work has been carried out by the author.
         \subsection*{Funding}
          Not applicable.

    \newpage

        \begin{figure}
         \begin{center}
        Stability and convergence of the two-stage explicit/implicit computational technique with spectral orthogonal basis Galerkin FEM.
         \begin{tabular}{c c}
         \psfig{file=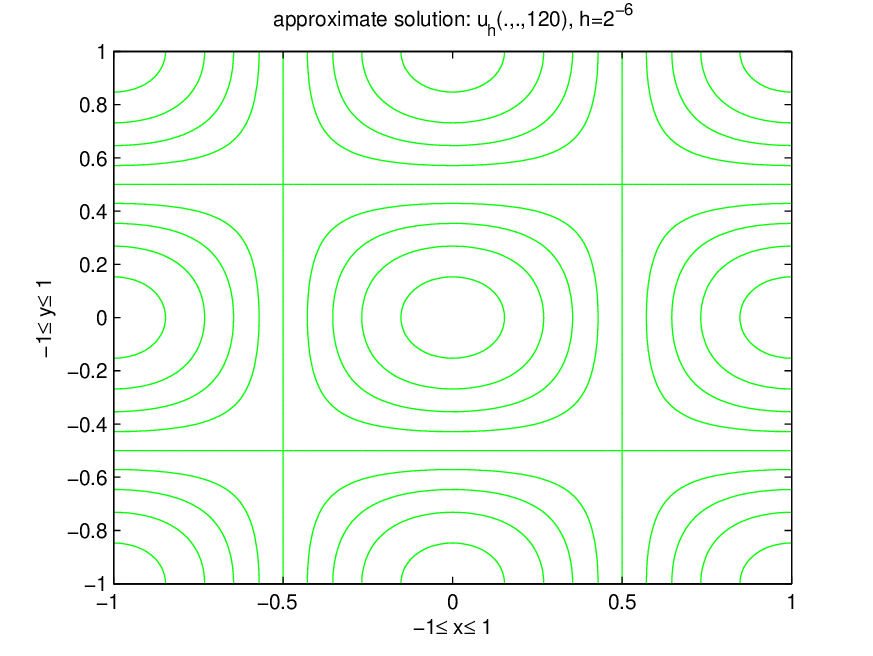,width=6cm} & \psfig{file=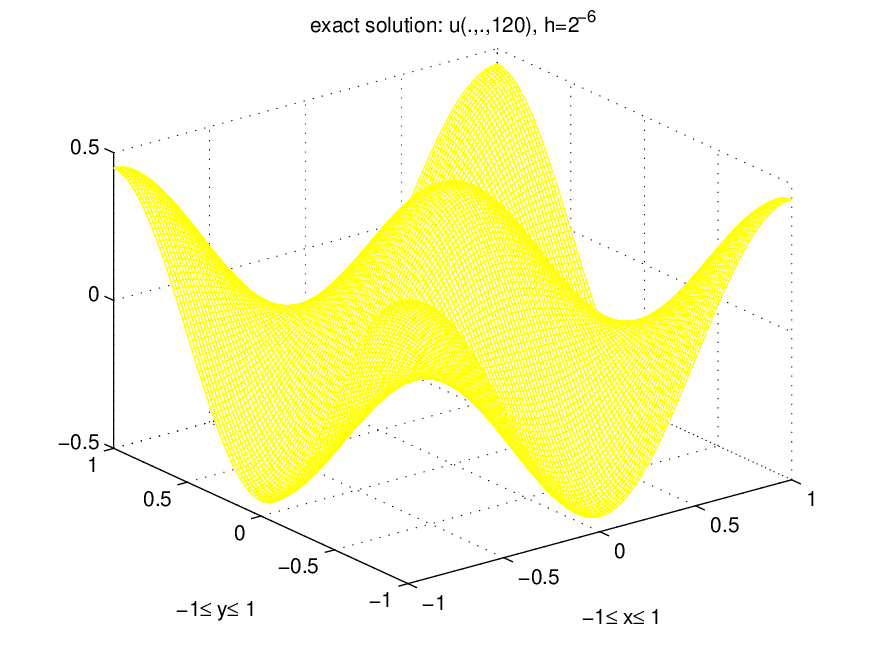,width=6cm}\\
         \psfig{file=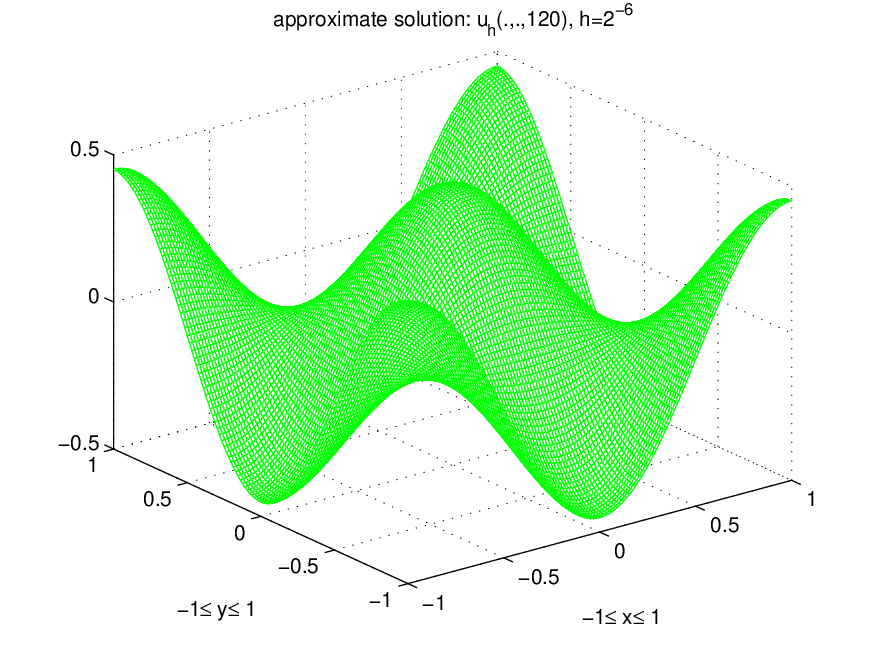,width=6cm} & \psfig{file=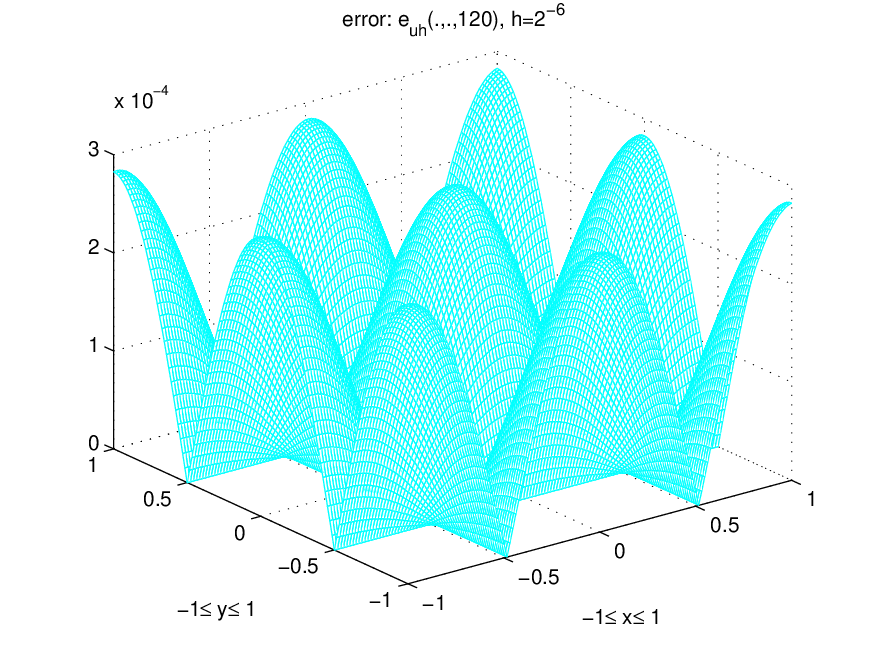,width=6cm}\\
         \psfig{file=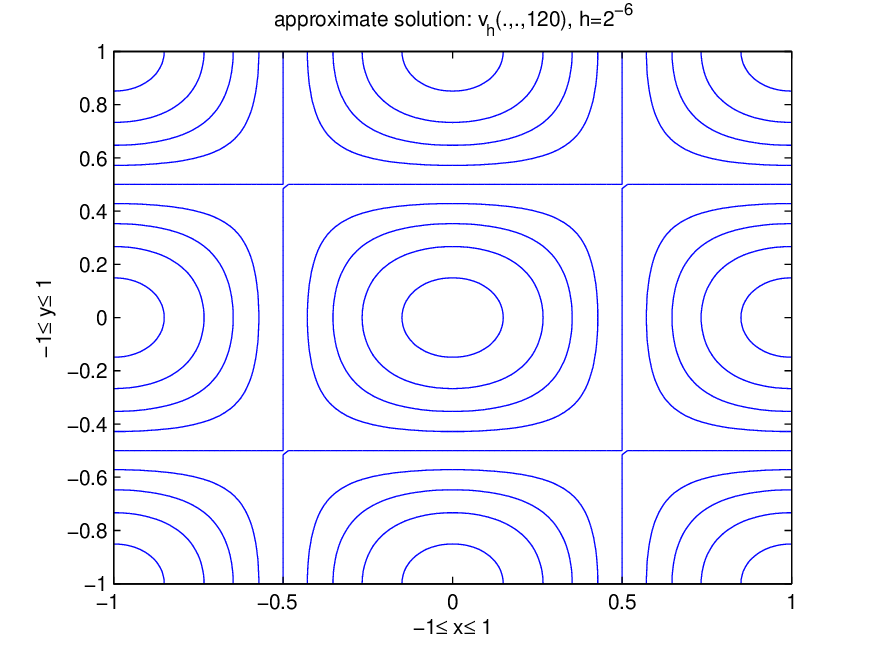,width=6cm} & \psfig{file=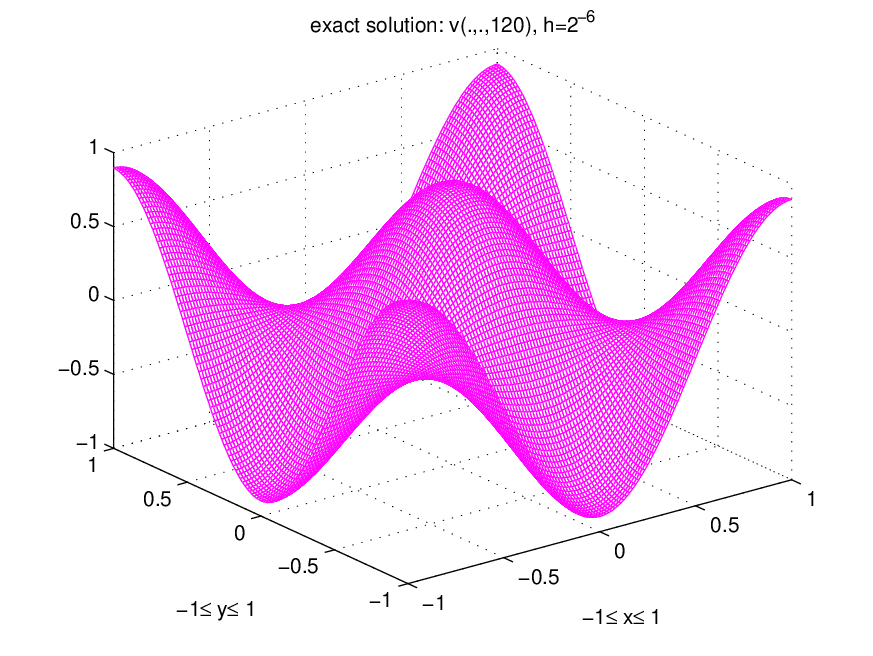,width=6cm}\\
         \psfig{file=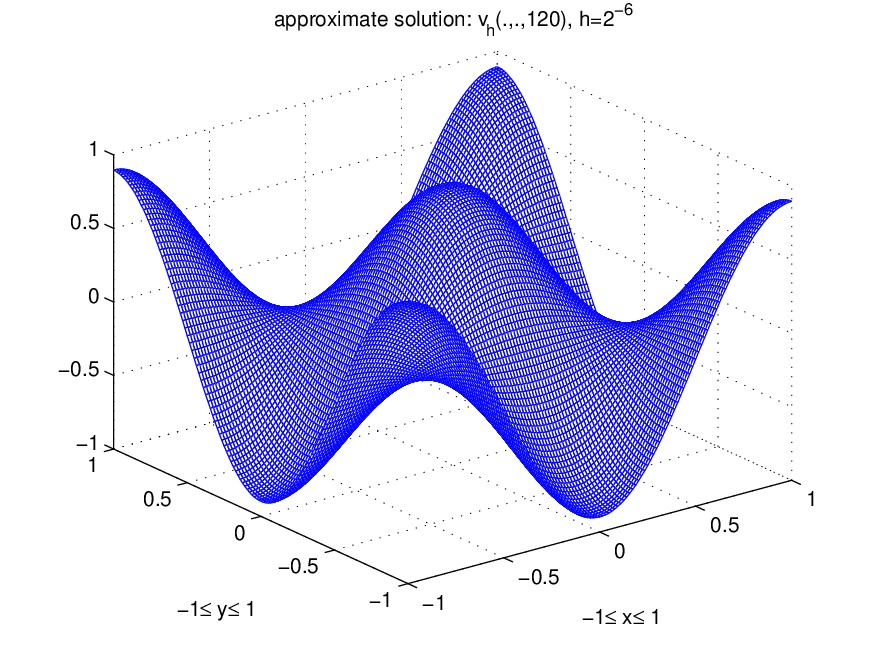,width=6cm} & \psfig{file=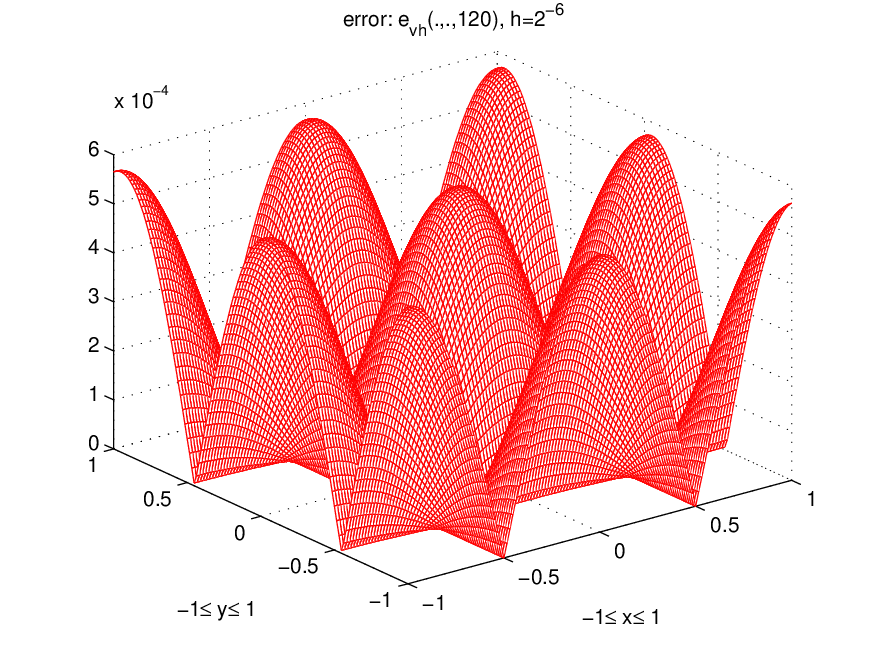,width=6cm}\\
         \end{tabular}
        \end{center}
        \caption{exact solution, approximate solution and error corresponding to Example 1}
        \label{figure1}
        \end{figure}

           \begin{figure}
         \begin{center}
        Stability and convergence of the two-stage explicit/implicit computational technique with spectral orthogonal basis Galerkin FEM.
         \begin{tabular}{c c}
         \psfig{file=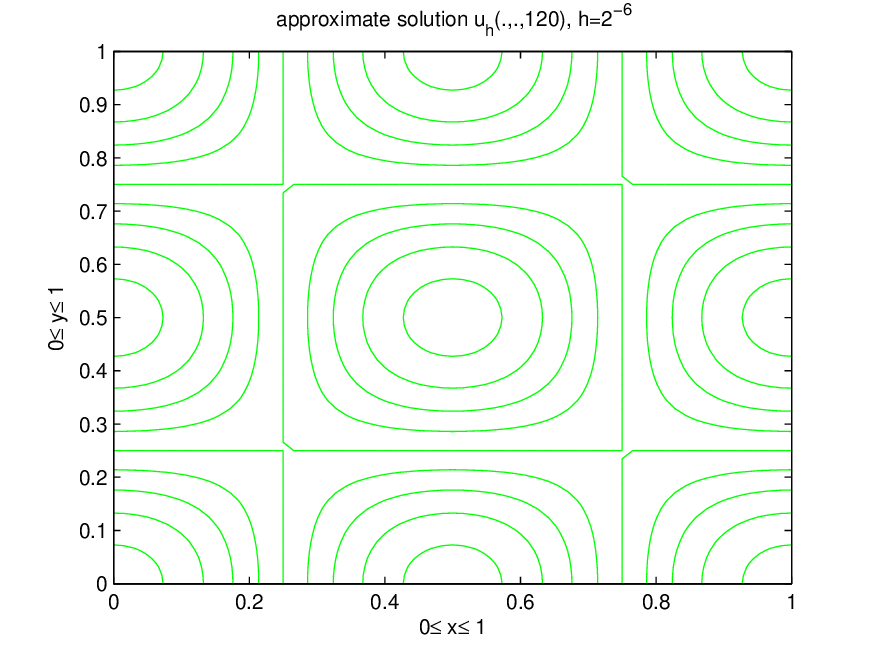,width=6cm}  & \psfig{file=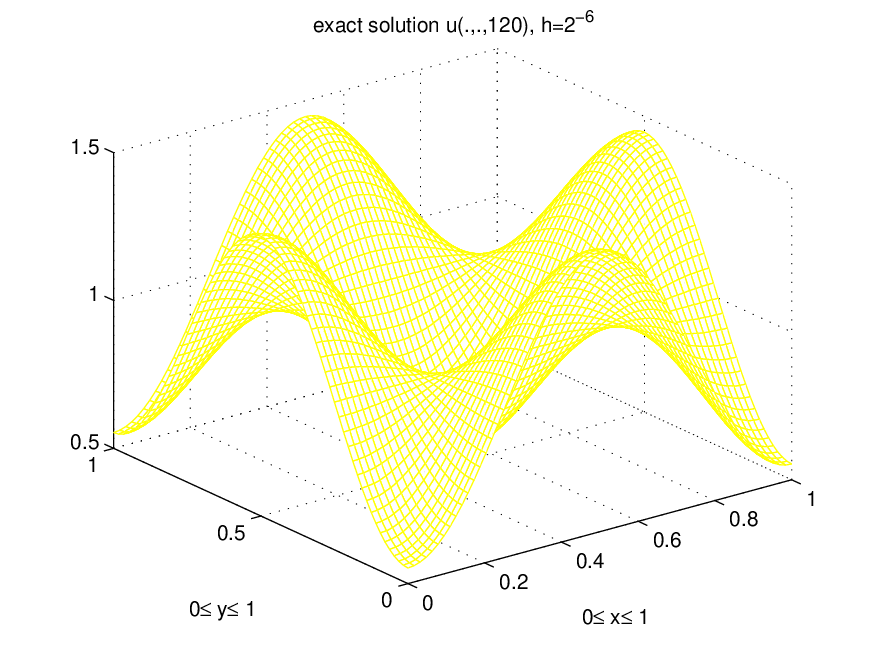,width=6cm}\\
         \psfig{file=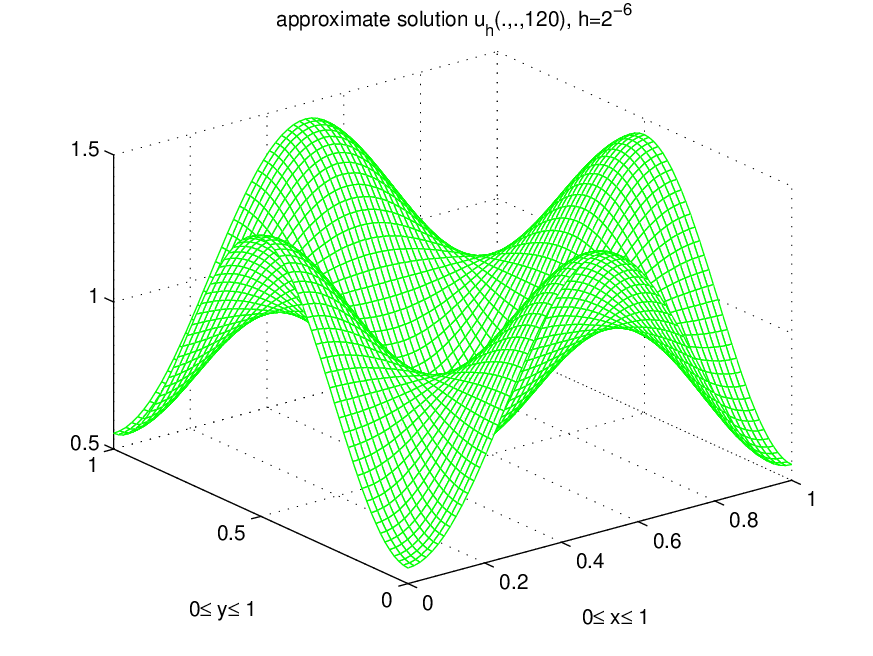,width=6cm} & \psfig{file=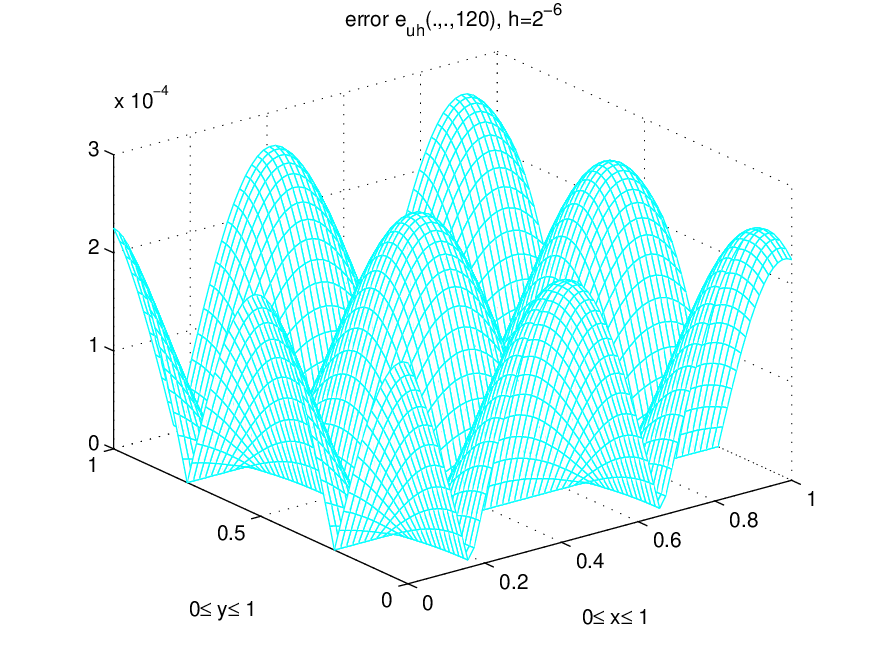,width=6cm}\\
         \psfig{file=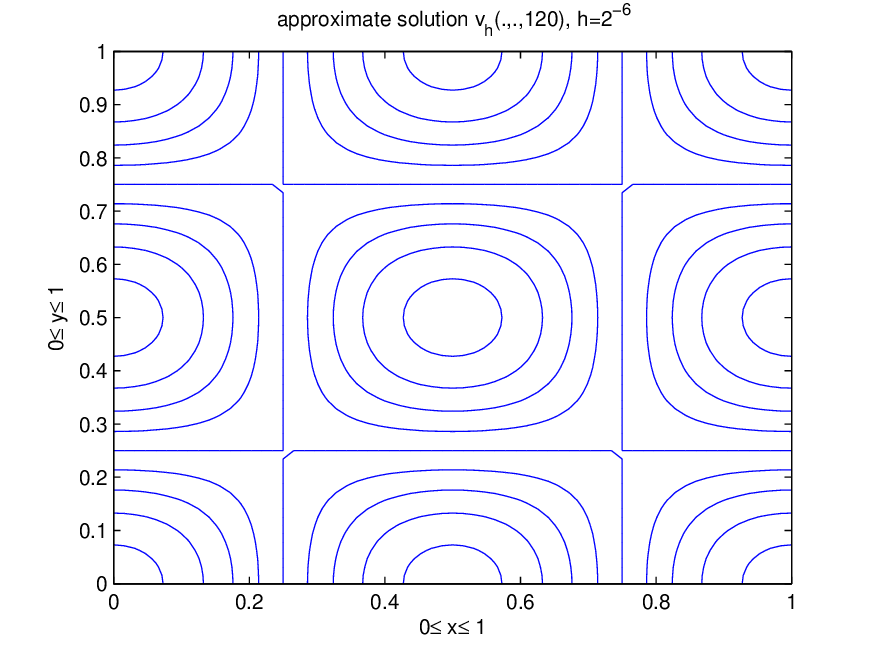,width=6cm} & \psfig{file=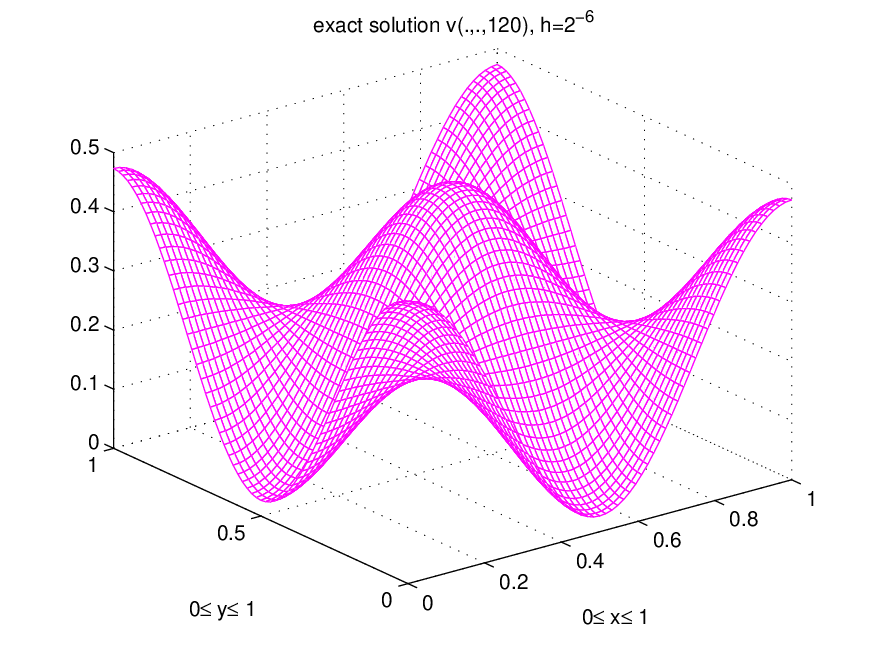,width=6cm}\\
         \psfig{file=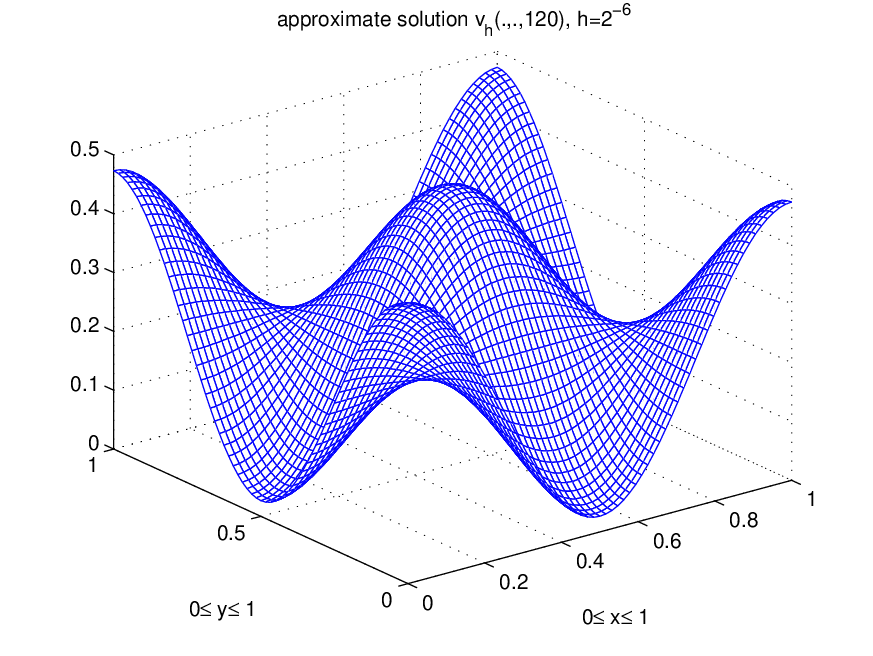,width=6cm} & \psfig{file=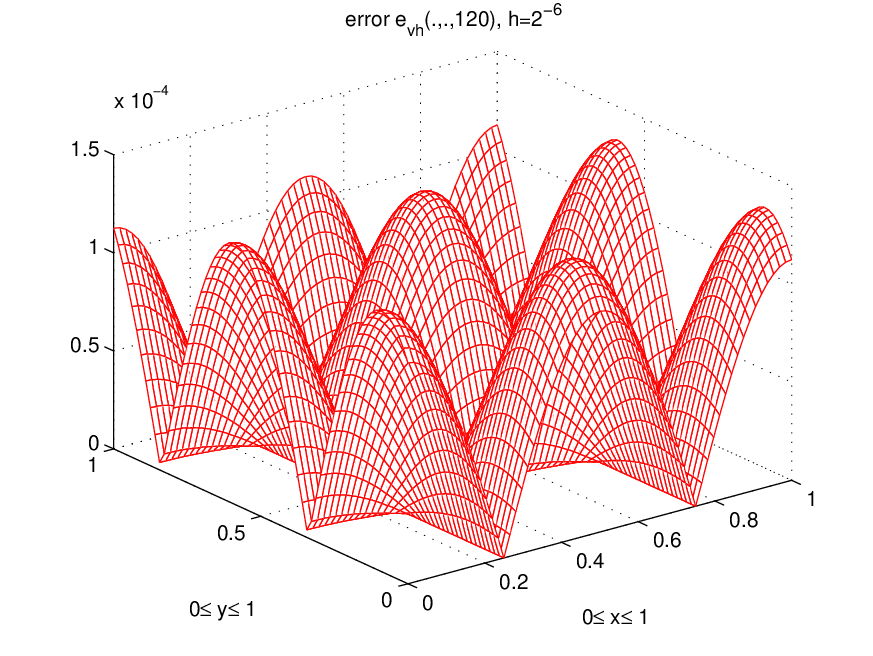,width=6cm}\\
         \end{tabular}
        \end{center}
        \caption{exact solution, approximate solution and error associated with Example 2}
        \label{figure2}
        \end{figure}

       \begin{figure}
         \begin{center}
        Stability and convergence of the two-stage explicit/implicit computational technique with spectral orthogonal basis Galerkin FEM.
         \begin{tabular}{c c}
         \psfig{file=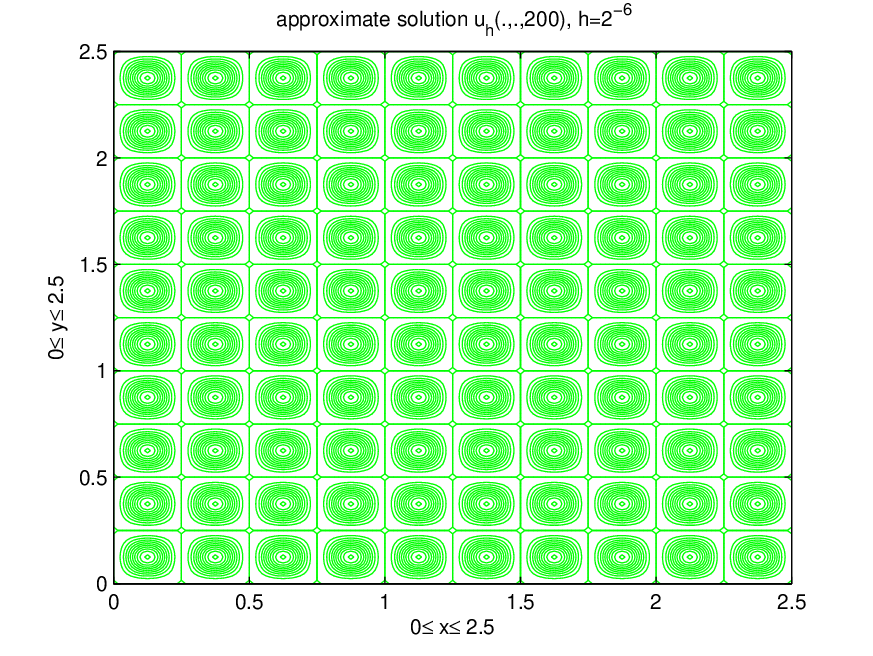,width=6cm} & \psfig{file=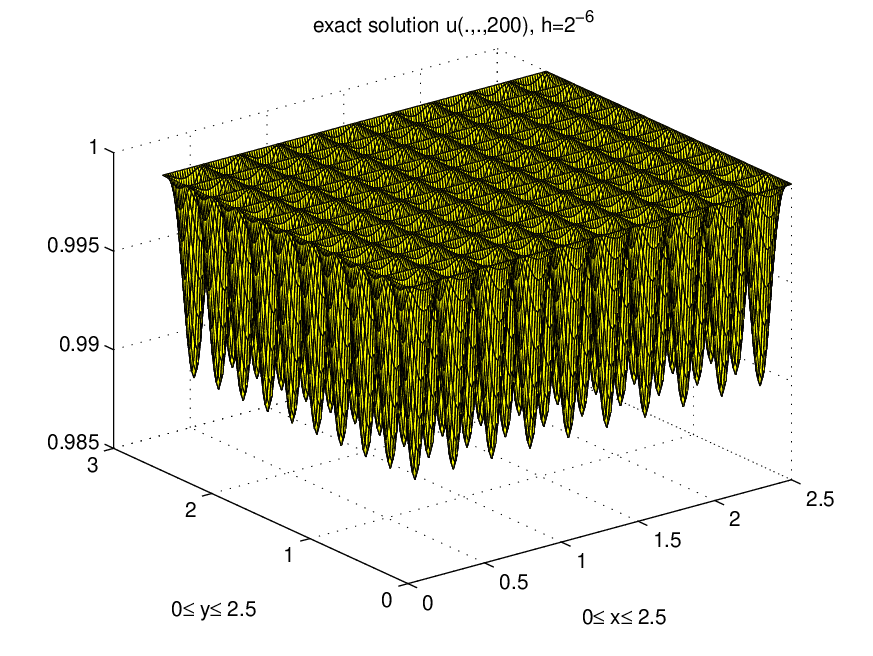,width=6cm}\\
         \psfig{file=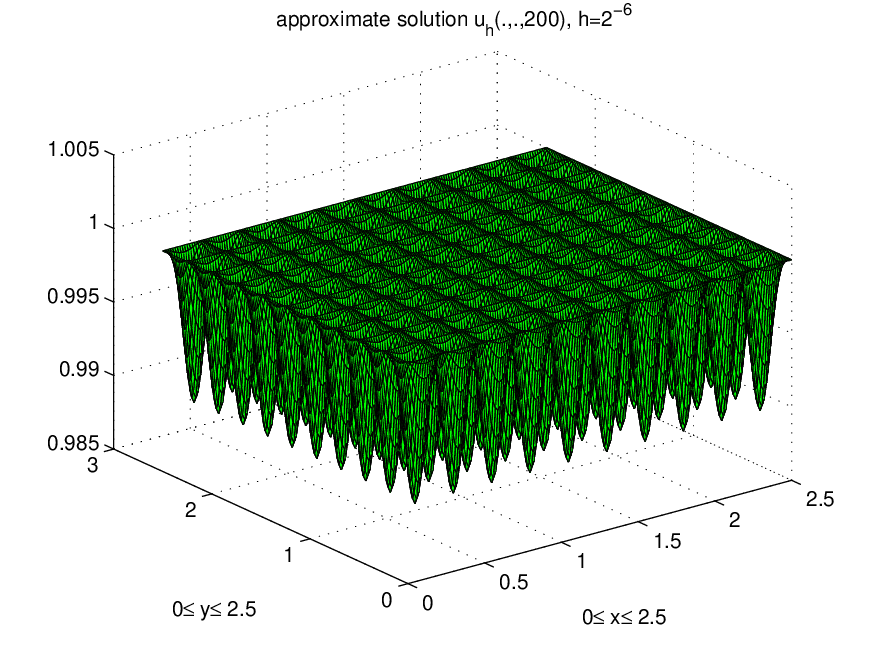,width=6cm} & \psfig{file=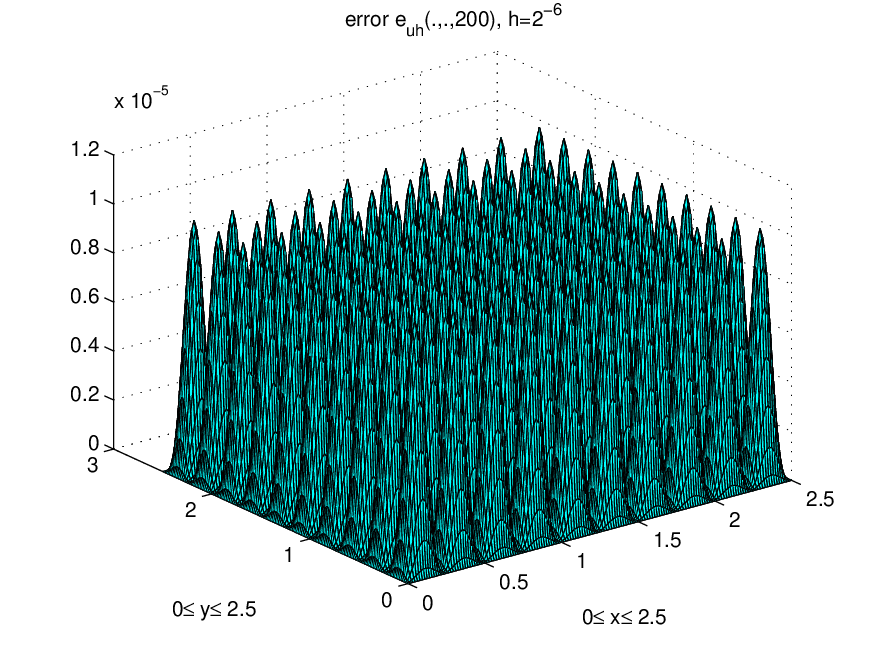,width=6cm}\\
         \psfig{file=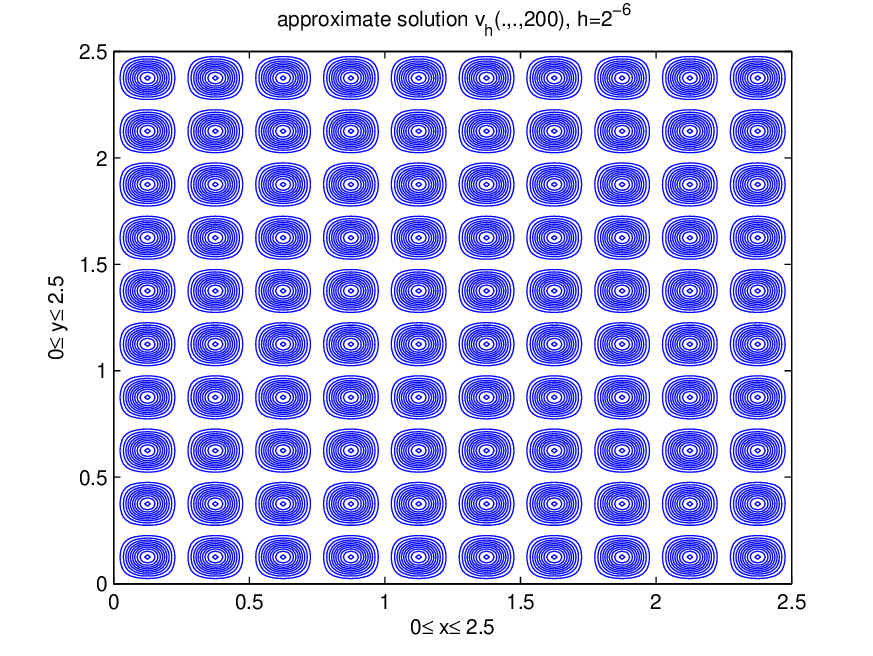,width=6cm} & \psfig{file=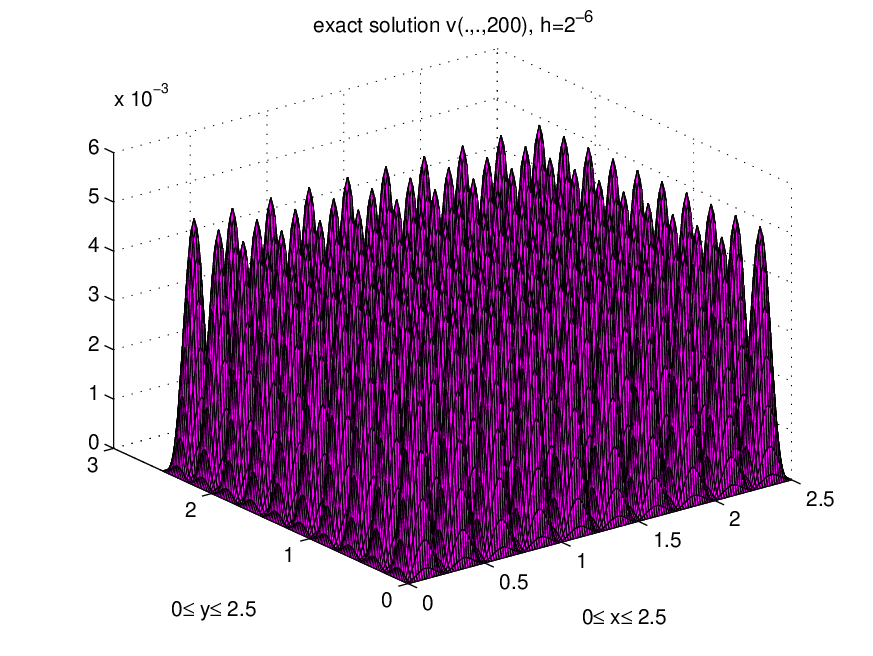,width=6cm}\\
         \psfig{file=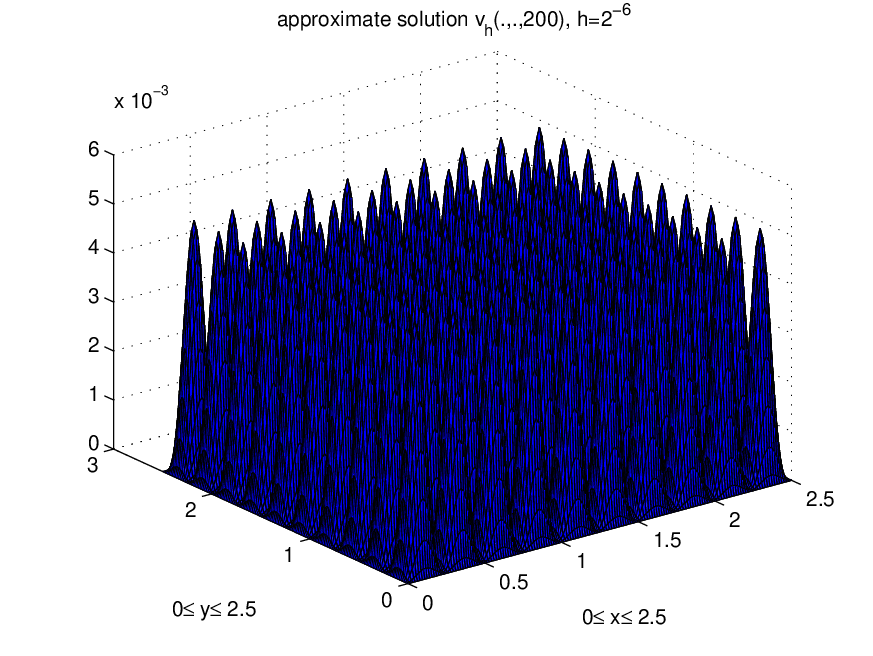,width=6cm} & \psfig{file=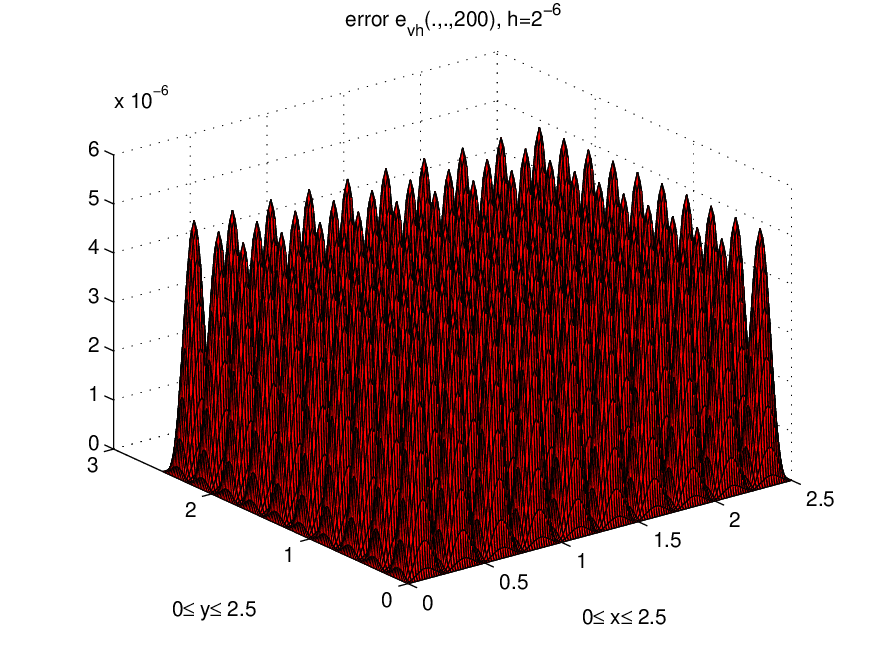,width=6cm}\\
         \end{tabular}
        \end{center}
        \caption{exact solution, approximate solution and error associated with Example 3}
        \label{figure3}
        \end{figure}

       \end{document}